\documentclass[11pt]{amsart}
\usepackage{graphicx}
\usepackage{xcolor}
\usepackage{amssymb,amsfonts,amsthm,amsmath,calligra, enumerate, tikz, stmaryrd}
\usepackage[utf8]{inputenc}
\usepackage{accents}
\usepackage{slashed}
\usepackage{yfonts}
\usepackage{mathrsfs,pifont}
\theoremstyle{definition}
\usepackage{tikz}
\usepackage[all]{xy}

\usepackage{pgfplots}
\pgfplotsset{compat=1.15}
\usepackage{mathrsfs}
\usetikzlibrary{arrows}

\usepackage[bookmarksnumbered=true, bookmarksopen=true]{hyperref}

\def\be{\begin{eqnarray}}
\def\ee{\end{eqnarray}}
\def\ben{\begin{eqnarray*}}
\def\een{\end{eqnarray*}}

\definecolor{wqwqwq}{rgb}{0.3764705882352941,0.3764705882352941,0.3764705882352941}
\definecolor{aqaqaq}{rgb}{0.6274509803921569,0.6274509803921569,0.6274509803921569}
\definecolor{yqyqyq}{rgb}{0.5019607843137255,0.5019607843137255,0.5019607843137255}
\definecolor{cqcqcq}{rgb}{0.7529411764705882,0.7529411764705882,0.7529411764705882}
\definecolor{uuuuuu}{rgb}{0.26666666666666666,0.26666666666666666,0.26666666666666666}

\usepackage{geometry}
\allowdisplaybreaks

\newcommand{\bC}{\mathbb{C}}

\newcommand{\bP}{\mathbb{P}}
\newcommand{\bQ}{\mathbb{Q}}
\newcommand{\bR}{\mathbb{R}}

\newcommand{\bZ}{\mathbb{Z}}

\newcommand{\bp}{\mathbf{p}}

\newcommand{\cA}{\mathcal{A}}

\newcommand{\cD}{\mathcal{D}}

\newcommand{\cK}{\mathcal{K}}
\newcommand{\cL}{\mathcal{L}}

\newcommand{\cN}{\mathcal{N}}
\newcommand{\cO}{\mathcal{O}}

\newcommand{\fX}{\mathfrak{X}}

\newcommand{\sA}{\mathsf{A}}
\newcommand{\sK}{\mathsf{K}}
\newcommand{\sT}{\mathsf{T}}

\newcommand{\loc}{\mathrm{loc}}

\newcommand{\vir}{\mathrm{vir}}

\newcommand{\Attr}{\operatorname{Attr}}

\newcommand{\Bl}{\operatorname{Bl}}

\newcommand{\Eff}{\operatorname{Eff}}

\newcommand{\ev}{\operatorname{ev}}

\newcommand{\Hom}{\operatorname{Hom}}

\newcommand{\Lie}{\operatorname{Lie}}

\newcommand{\pt}{\operatorname{pt}}

\newcommand{\Spec}{\operatorname{Spec}}

\newcommand{\tr}{\operatorname{tr}}
\newcommand{\eff}{\operatorname{eff}}

\theoremstyle{definition}
\newtheorem{Definition}{Definition}[section]

\newtheorem{Remark}[Definition]{Remark}
\newtheorem{Example}[Definition]{Example}
\numberwithin{equation}{section}

\theoremstyle{Theorem}
\newtheorem{Theorem}[Definition]{Theorem}

\newtheorem{Lemma}[Definition]{Lemma}
\newtheorem{Corollary}[Definition]{Corollary}

\begin{document}
	
	\title{Quantum $K$-theory of toric varieties, level structures, and 3d mirror symmetry}
	\author{Yongbin Ruan, Yaoxiong Wen,  Zijun Zhou}
	\date{}
	\maketitle
	\thispagestyle{empty}

	\setlength{\parskip}{1ex}

\begin{abstract}
	We introduce a new version of 3d mirror symmetry for toric stacks, inspired by a 3d $\cN = 2$ abelian mirror symmetry construction in physics. Given some toric data, we introduce the $K$-theoretic $I$-function with effective level structure for the associated toric stack. When a particular stability condition is chosen, it restricts to the $I$-function for the particular toric GIT quotient. The mirror of a toric stack is defined by the Gale dual of the original toric data. We then proved the mirror conjecture that the $I$-functions of a mirror pair coincide, under the mirror map, which switches K\"ahler and equivariant parameters, and maps $q\mapsto q^{-1}$. 
\end{abstract}

	\tableofcontents

\vspace{1cm}

\section{Introduction}

One of the recent remarkable discoveries is the connection between quantum $K$-theory and 3d TQFT.  For a long time, quantum $K$-theory has been viewed as a variant to quantum cohomology, which comes from a 2d TQFT. The new connection puts quantum $K$-theory in a different path
to the new territory of 3d physics, the mathematics of which was much less understood. The 3d physics has its own mirror symmetry phenomenon, which falls into two versions, for $\cN=4$ theories versus $\cN=2$ theories. 3d $\cN=4$ theories are much better behaved in physics, but apply to more restrictive targets such as Nakajima quiver varieties, due to the presence of more supersymmetries.  There are many mathematical results on the enumerative-geometric aspect of 3d $\cN=4$ mirror symmetry by Okounkov's group \cite{AOelliptic, KZ, RSVZ, RSVZ2, SZ}.

Our main interest is in 3d $\cN=2$ theories which apply to a general K\"ahler manifold. Its mirror symmetry was poorly understood even in physics. In addition, for such theories there is a new feature called \emph{level structure} introduced by Ruan--Zhang \cite{RZ}.  Any duality for $\cN=2$ theories should incorporate the level structure, which makes it quite difficult. Right now, we are still in the early stage of exploration. There are many conjectural examples for 3d $\cN=4$ mirror pairs in physics. However, as far as our knowledge of physics literature \cite{DT, AHKT, ARW}, the only major class of 3d $\cN=2$ mirror pairs are toric varieties as follows.

Consider the following short exact sequence
\begin{align} \label{exact-intro}
\xymatrix{
	0 \ar[r] & \mathbb{Z}^k \ar[r]^-\iota &  \mathbb{Z}^n \ar[r]^-\beta &  \mathbb{Z}^{n-k} \ar[r] & 0. 
} 	
\end{align}

{\bf 3d Toric Mirror Conjecture: }
\emph{ 
\begin{itemize}
	\setlength{\parskip}{1ex}
\item {\bf Model A:} Let $\iota=(\iota_{ij})$ for $j=1, \cdots, k$ and $i=1, \cdots, n$. One side of the mirror symmetry is the toric quotient $[{\bf C}^n/({\bf C^*})^k]$ defined by the charge matrix $\iota$, with the so-called {effective level } (see Definition \ref{effective-level}) $\frac{1}{2} \sum_{j,l=1}^k  \iota_{ij} \iota_{il}$. We consider its equivariant theory with quantum parameters $\zeta^{b}$ for $b=1, \cdots, k$ and the equivariant parameters $m_{i}$ for $i=1, \cdots, n$.
\item {\bf Model B:} Let $\beta=(\beta_{ji})$ for $j=1, \cdots, n-k$ and $i = 1, \cdots, n$. The other side of mirror is the toric quotient $[{\bf C}^n/({\bf C^*})^{n-k}]$ defined by the charge matrix $\beta^T$, with the effective level $\frac{1}{2}\sum_{j,l=1}^{n-k} \beta_{ji} \beta_{li}$. We consider its equivariant theory with quantum parameters $\widehat{\zeta}^{p}$ for $p=1, \cdots, n-k$ and equivariant parameters  $\widehat{m}_{i}$ for $i=1, \cdots, n$.
\end{itemize}
Model A is ``equivalent" to Model B via the mirror map
\begin{align}
&\zeta^{b}-\frac{1}{2} \sum_{i=1}^{n} \iota_{ib} m_{i}=\sum_{i=1}^{n} \iota_{ib} \widehat{m}_{i} \label{mirror-map-1}  \\
&-\sum_{i=1}^{n} \beta_{pi} m_{i}=\widehat{\zeta}^{p}+\frac{1}{2} \sum_{i=1}^{N} \beta_{pi} \widehat{m}_{i} . 	\label{mirror-map-2}
\end{align} 
Note that the above statement is slightly different from that in \cite{AHKT}.}

One can study the above ``equivalence" in different contexts. The main goal of this article is to prove the conjecture for equivariant quantum $K$-theoretic $I$-functions. For a person with background in 2d mirror symmetry, the above conjecture is rather strange, since in 2d mirror symmetry, a toric variety is mirrored to a Landau--Ginzburg model, instead of  another toric variety. Furthermore, there are many chambers of toric quotients for given toric data, which in general admit different cohomology/$K$-groups. A logical first step, which is the first approach we tried, seems to be an attempt to match different chambers of mirror pairs. Unfortunately, for the simplest example of projective spaces, this turns out to be false! Our main conceptual breakthrough is the realization that we should consider the $I$-function for the \emph{entire toric stack} (see Definition \ref{total-stack-I-intro} and \ref{total-stack-I}), rather than any of its particular GIT quotients, in the sense to sum up the contributions from ALL chambers. Of course, a naive sum would result in double counting. At this point, we do not know how to do this in general. 

For the equivariant theory of toric variety, its $I$-function (defined by quasimap graph moduli space) localizes to the fixed point contributions. It is easy to see from examples that a fixed point could appear in multiple chambers. 

Let $\sK := (\bC^*)^k$, and $\sT := (\bC^*)^n$. Let $\fX$ be the quotient stack $ [ \bC^n / \sK  ]$. Choose a character $\theta \in \Lie_{\bR}(\sK^\vee) $. We obtain a GIT quotient $ X_{\theta} $, and some of the $\sT$-fixed points $\bp \in \fX $ descend to the fixed points $\bp \in X_{\theta}$. Combinatorially, the $\sT$-fixed points of $\fX$ can be identified with subsets $\bp \subset \{1, \cdots, n\}$ of size $k$, where the $k\times k$ submatrix of $\iota$ corresponding to $\bp$ is of full rank. 

Let $I_{\bp, \theta}$ be the fixed point contribution of $\bp$ to the $I$-function of $X_{\theta}$.  A key technical lemma is

\begin{Lemma}[Corollary \ref{independence-I-bp}]
$I_{\bp, \theta}$ is independent of the chamber and we denote it by $I({\bp})$.
\end{Lemma}

The above lemma leads to the following definition. 

\begin{Definition} \label{total-stack-I-intro}
We introduce the $I$-function for the toric stack $\fX$ with the effective level as follows
$$
I^{\eff} (\fX) :=  \sum_{\bp \in \fX^{\sT} } I^{\eff}({\bp}).	
$$
Let $\bp \in \fX $ be a fixed point, the \emph{modified $I$-function with the effective level}  is defined as 
$$
\widetilde I^{\eff} (\bp) := e^{-\sum_{i\not\in \bp} \frac{\ln z_i \ln U_i |_\bp }{\ln q} } \cdot  \prod_{i\not\in \bp} \frac{1-U_i^{-1}|_{\bp}}{ (qU_i |_\bp )_\infty } \cdot  I^{\eff}  (\bp). 
$$
Similarly, we define 
$$
\widetilde{I}^{\eff}(\fX) := \sum_{\bp \in \fX^{\sT}} \widetilde{I}^{\eff}(\bp). 
$$
\end{Definition}

The exponential prefactor $e^{-\sum_{i\not\in \bp} \frac{\ln z_i \ln U_i |_\bp }{\ln q} }$ here is a $q$-analogue of the exponential factor $e^{\sum_i \frac{t_i p_i}{z}}$ as in the cohomological $I$-functions. The factors $\dfrac{1}{(q U_i |_\bp )_\infty}$ are $q$-analogues of gamma functions. The $Q$-coefficients of $\widetilde I^{\eff} (\bp)$ no longer lie in $K_{\sT \times \bC_q^*} (\bp)_\loc$. 

We consider the dual exact sequence of (\ref{exact-intro}) as the \emph{mirror}:
\begin{equation} \label{dual-exact-intro}
\xymatrix{
	0 \ar[r] & (\bZ^d)^\vee \ar[r]^-{\iota^!} & (\bZ^n)^\vee \ar[r]^-{\beta^!} & (\bZ^k)^\vee \ar[r] & 0,
} 
\end{equation}
where 
$$
\iota^! := \beta^T, \qquad \beta^! := \iota^T. 
$$
This is often refered to as the Gale dual. 

\begin{Definition}[Definition \ref{defn-mirror}] 
	
Let $\sA := \sT / \sK \cong (\bC^*)^d$ be the quotient torus. 
	
	\begin{itemize}
		
		\setlength{\parskip}{1ex}
		
		\item The \emph{mirror toric stack} to $\fX$ is defined as
		$$
		\fX^! := [\bC^n / \sA^\vee], 
		$$
		the toric quotient stack associated with the short exact sequence (\ref{dual-exact-intro}), where the action by $\sA^\vee$ is defined by $\iota^!$. 
		
		\item There is a natural bijection between fixed points of a mirror pair. Given a fixed point $\bp$ of $\fX$, the \emph{mirror fixed point} is defined as the complement
		$$
		\bp^! = \{1, \cdots, n \} \backslash \bp \quad \in \quad (\fX^!)^\sT . 
		$$
		
	\end{itemize}
	
\end{Definition}

The main result of this article is a proof of above mirror conjecture for modified $I$-functions with effective level structure for toric stacks.

\begin{Theorem} [Theorem \ref{main-theorem}]
	Let $(\fX, \fX^!)$ be a mirror pair of toric stacks. Let $q^{z_i \partial_{z_i}}$ denote the $q$-difference operator that shifts $z_i \mapsto q z_i$, and similar with $q^{a_i \partial_{a_i}}$. 
	
	\begin{enumerate} [1)]
		\setlength{\parskip}{1ex}
		
		\item The modified $I$-function of $\fX$ with effective level structure satisfies the following two sets of $q$-difference equations, with respect to the K\"ahler and equivariant parameters. 
		
		\begin{itemize} 
			\setlength{\parskip}{1ex}
			
			\item Let $\{ e_i \}_{i=1}^n$ be the standard basis of $\bZ^n$, and consider any $\sum_{i=1}^n \mu_i e_i \in  \ker \beta $ such that $\mu_i = \pm 1  $ or $0$.  Denote by $S_\pm$ the subset of indices with $\mu_i= \pm1$. Then  
		$$
			\left[ \prod_{i\in S_+} ( z_i^{-1} (1- q^{- z_i \partial_{z_i}} ) ) - \prod_{i\in S_-} (z_i^{-1} (1- q^{- z_i \partial_{z_i}}  ) ) \right] \widetilde I^{\eff} (\fX) = 0. 
			$$
			
			\item Let $\{ e^!_i \}_{i=1}^n$ be the standard basis of $\bZ^n$ in the dual exact sequence, and consider any $\sum_{i=1}^n \mu^!_i e^!_i \in  \ker \iota^T $ such that $\mu^!_i = \pm 1  $ or $0$. Denote by $R_\pm$ the subset of indices with $\mu^!_i= \pm 1$. Then 
		$$
			\left[ \prod_{i\in R_+} (a_i^{-1} (1- q^{ a_i \partial_{a_i}} ) ) -  \prod_{i\in R_-} ( a_i^{-1} (1- q^{ a_i \partial_{a_i}} ) ) \right] \left( e^{\sum_{i=1}^n \frac{\ln z_i \ln a_i}{\ln q} } \cdot \widetilde I^{\eff} (\fX) \right) = 0 . 
		$$
			
		\end{itemize}
		Moreover, the solution to the above difference equations is unique, with certain prescribed asymptotic initial condition (see Lemma \ref{asymptotic}). 
		
		\item Under the mirror map 
		$$
		\tau (z_i^!) = a_i, \qquad \tau (a_i^!) = z_i, \qquad \tau (q) = q^{-1},  
		$$
		the two sets of $q$-difference equations (\ref{eqn-for-kalher}) (\ref{eqn-for-equiv}), for modified $I$-functions of the mirror pair $(\fX, \fX^!)$ with effective level structure, coincide with each other. Therefore, combining it with the uniqueness result, we have
		$$
		\widetilde I^{\eff}(\fX) = e^{\sum_{i=1}^n \frac{\ln z_i \ln a_i}{\ln q} } \cdot \tau ( \widetilde I^{\eff} (\fX^!) )	. 
		$$
		
	\end{enumerate}

\end{Theorem}

\begin{Remark}
Our mirror map $\tau$ appears different from the mirror map (\ref{mirror-map-1}) and (\ref{mirror-map-2}) in physics literature. In Section \ref{parameters}, we introduced two versions of K\"ahler parameters, the \emph{effective} and the \emph{redundant}. Our mirror map $\tau$ is written in terms of the redundant parameters. 

To see how they are related, first rewrite (\ref{mirror-map-1}) (\ref{mirror-map-2}) as
\begin{align} \label{rewrite-mirror-map}
\zeta^{b} =\sum_{i=1}^{n} \iota_{ib} \left( \widehat{m}_{i} +\frac{1}{2}  m_{i} \right), \qquad  - \sum_{i=1}^{n} \beta_{pi} \left( \frac{1}{2}\widehat{m}_{i}+ m_{i} \right) =\widehat{\zeta}^{p}	. 
\end{align} 
$\zeta^b$ can then be interpreted as our effective K\"ahler parameter $Q_b$ (see Section \ref{parameters}). From Remark \ref{effective-Kahler} we know the relation between effective and redundant K\"ahler parameters are as follows
\begin{align*}
\ln Q_b = \sum_{i=1}^n \iota_{ib} \ln z_i, \qquad	 \ln Q^!_p = \sum_{i=1}^n \beta_{pi} \ln z^!_i . 
\end{align*}
Applying mirror map (\ref{mirror-map}), we obtain
\begin{align*}
\tau(\ln Q_b) = \sum_{i=1}^n \iota_{ib} \ln a^!_i, \quad \tau(\ln Q^!_p) = \sum_{i=1}^n \beta_{pi} \ln a_i	, 
\end{align*}
which is the same as (\ref{rewrite-mirror-map}) if we let 
\begin{align*}
\widehat{m}_{i} +\frac{1}{2}  m_{i}=  \ln a^!_i, \quad 	\frac{1}{2}\widehat{m}_{i}+ m_{i} = \ln a_i^{-1} . 
\end{align*}	 
\end{Remark}

The paper is organized as follows. In Section \ref{sec-2}, we briefly outline the duality of combinatorial structures for mirror pairs. This section is an extension of previous results of the last author \cite{SZ}. The main technical results are in Section \ref{sec-3} where we prove the main technical lemma and compute the modified $I$-function with effective level structure. The proof of the main theorem is in the Section \ref{sec-4}. To give the reader a sense of the problem, we also show a hands-on approach to the case of projective space, using the $q$-binomial formula. The general case follows from the duality of $q$-difference equations and the analysis of their solution spaces.

\subsection{Acknowledgement}
The bulk of this work was done during our stay at the Institute for Advanced Study in Mathematics at Zhejiang University. We express our special thanks to the institute for the wonderful environment and support. The second author would like to thank Prof. Bohan Fang, Prof. Huijun Fan and Peking University for the helpful support during the visit. Thanks are also due to Prof. Shuai Guo and Ming Zhang for helpful discussions.

\vspace{1cm}

\section{Toric stacks and GIT quotients from fixed points} \label{sec-2}

Toric varieties \cite{Ful} have been studied for decades and provide important examples in algebraic geometry that can be explicitly described in terms of combinatorial data. To define a toric variety, one starts with a fan $\Sigma$ in a lattice $N \cong \bZ^r$, whose rays are denoted by $\rho_1, \cdots, \rho_n$. The toric variety is then constructed as a quotient $(\bC^n - Z ) / \sK$, where $Z$ is the irrelavant locus determined by the fan, and $\sK \cong (\bC^*)^k$ (we denote $k: = n-r$) is a torus acting on $\bC^n$, whose action is determind by the relations among the rays $\rho_i$'s. 

An alternative way to construct toric varieties is to consider them as (real) symplectic reductions, or equivalently, GIT quotients. By choosing an appropriate stability condition $\theta$, the irrelevant subvariety $Z$ turns out to be the unstable locus of the action, and the quotient stated as above is the GIT quotient $\bC^n /\!/_\theta \sK$. However, there are more possible choices of $\theta$ that might be interesting. The variation of GIT \cite{DH, Tha} implies that when $\theta$ crosses a wall and enters a different chamber, one obtains a different GIT quotient, which might not be the toric variety defined by the original fan $\Sigma$. Even if one happens to obtain the same variety, it is not canonically isomorphic to the original one, but related to it by a birational transformation. Therefore, if we are interested in a global understanding of \emph{all possible} GIT quotients, it is better to study the quotient stack $[\bC^n / \sK]$ directly. This is the viewpoint we would take in the following of this paper. 

Let $\sT := (\bC^*)^n$ be the standard $n$-dimensional torus acting on $\bC^n$. The action \footnote{which is, by construction, faithful. } by $\sK$ is then characterized by an injective homomorphism $\sK \to \sT$, or equivalently, an injective homomorphism of free $\bZ$-modules $\iota: \bZ^k \to \bZ^n$. The map $\iota$ will be our starting datum. A more convenient and symmetric way is to consider the short exact sequence
\begin{equation} \label{knd}
\xymatrix{
	0 \ar[r] & \bZ^k \ar[r]^-\iota & \bZ^n \ar[r]^-\beta & \bZ^{r} \ar[r] & 0. 
} 
\end{equation}

\begin{Definition}
A matrix $\iota$ is called \emph{totally unimodular} if the determinants of all its maximal square submatrices are either $\pm 1$ or $0$. 
\end{Definition}

\begin{Remark}
Let $\iota$ be the $n\times k$ matrix as above, which is of rank $k$. The definition of totally unimodularity is equivalent to the following: there exists $P\in GL(k, \bZ)$, such that the determinants of all square submatrices (of any size) of $\iota P$ are either $\pm 1$ or $0$. In particular, all entries of $\iota P$ are $\pm 1$ or $0$. 
\end{Remark}

In the rest of this paper, we will always assume that $\iota$ is totally unimodular. It then follows that $\beta$ is also totally unimodular. 

Consider the quotient stack 
$$
\fX := [\bC^n / \sK],
$$
where the action of $\sK$ on $\bC^n$ is defined by $\iota$. The action of $\sT$ on $\bC^n$ descends to $\fX$, and to any of its GIT quotients $X$. The existence of this torus action enables us to study the geometry from the aspect of $\sT$-equivariant theory. Moreover, this torus is large enough, in the sense that $X$ is a GKM manifold \cite{GKM}, whose $\sT$-equivariant geometry can be completely recovered from the information of its $\sT$-fixed points and the 1-dimensional $\sT$-orbits connecting them. For the quotient stack $\fX$, its $\sT$-fixed points are characterized as follows. 

\begin{Lemma}
A $\sT$-\emph{fixed point} of $\fX$ is given by a subset $\bp \subset \{1, \cdots, n\}$, such that the $i$-th rows of the matrix $\iota$ with $i\in \bp$ are linearly independent. 
\end{Lemma}

\begin{proof}
	Geometrically, the locally closed subset $\{ x\in \bC^n \mid  x_i \neq 0, \ i\in \bp; x_i = 0, \ i\not\in \bp \} \cong (\bC^*)^k$ in $\bC^n$ defines a closed substack in $\fX$ which is isomorphic to $[(\bC^*)^k / \sK] \cong \pt$, and invariant under the $\sT$-action, in the sense of \cite{Rom}. 
\end{proof}

By abuse of notation, we will also denote this closed substack by $\bp$, and write $\bp\in \fX^\sT$. 

\subsection{K\"ahler cone and GIT quotients}

Let $\Lie_\bR (\sK^\vee) \cong \bR^k$ be  the (real) Lie algebra of the character group of $\sK$, which can be identified with the space of stability conditions one can take when performing GIT quotients. Given any $\theta \in \Lie_\bR (\sK^\vee)$, the GIT theory \cite{MFK} defines an open subscheme in $\fX$:
$$
X_\theta := \bC^n /\!/_\theta \sK. 
$$
By results on the variation of GIT, the space of stability conditions $\Lie_\bR (\sK^\vee)$ admits a wall-and-chamber structure. In other words, there exist certain walls (i.e. codimension one subsets) in $\Lie_\bR (\sK^\vee)$, the connected components of whom are called chambers, such that when we vary $\theta$ in a single chamber, the resulting space $X_\theta$ stays the same. 

In this paper we will only consider the case where $\theta$ is chosen \emph{generically}, i.e. avoiding all the walls. By the totally unimodularity of $\iota$, if nonempty, $X_\theta$ obtained for generic $\theta$ is always a smooth toric variety of dimension $d$. 

The action by the torus $\sT$ descends naturally to the quotient. However, not all fixed points $\bp$ of $\fX$ lie in the quotient $X_\theta$. Recall that for generic $\theta$, the GIT quotient can be defined as $X_\theta = \bC^{n,s} / \sK$, where $\bC^{n,s} \subset \bC^n$ is the stable locus, determined by the stability condition $m$. It might happen that representatives for $\bp$ in $\fX$ fall in the unstable locus, and hence $\bp$ is excluded in the GIT quotient procedure. 
\begin{Definition} \label{Kahler-cone}
	Let $\bp\in \fX$ be a $\sT$-fixed point. The \emph{K\"ahler cone} $\cK(\bp)$ associated with $\bp$ is defined as 
	$$
	\cK(\bp) := \{ \theta \in \Lie_\bR (\sK^\vee) \mid  \bp \in X_\theta \}. 
	$$
	We also define the \emph{effective cone} $\Eff (\bp) := \cK(\bp)^\vee$ as the dual of $\cK(\bp)$, which lies in $\Lie_\bR (\sK)$. 
\end{Definition}

The following lemma provides a combinatorial description of $\cK(\bp)$. 

\begin{Lemma}
	$\cK(\bp)$ is the interior of the cone in $\bR^k$ generated by the $i$-th rows of $\iota$ with $i\in \bp$. 
\end{Lemma}

\begin{proof}
	Let $m\in \Lie_\bR (\sK^\vee)$ be a stability condition. To tell whether $p\in X_m$, or equivalently, to tell whether a representative $x\in \bC^n$ of $\bp$ is stable, we apply the Hilbert--Mumford criterion. It states that $x \in \bC^n$ is stable, if and only if for any 1-parameter subgroup $\lambda: \bC^* \to \sK$, either the limit $\lim_{t\to 0} \lambda (t) \cdot x$ does not exist, or $\langle \lambda, \theta \rangle > 0$.
	
	A general 1-parameter subgroup $\lambda : \bC^* \to \sK$ is of the form $\lambda (t) = (t^{\lambda_1}, \cdots, t^{\lambda_k} )$, with $\lambda_1, \cdots, \lambda_k \in \bZ$. It acts on $x\in \bC^n$ as
	$$
	\lambda (t) \cdot x = \left( t^{\sum_{j=1}^k \iota_{1j} \lambda_j } x_1 , \cdots, t^{ \sum_{j=1}^k \iota_{nj} \lambda_j } x_n \right) . 
	$$
	A representative of $\bp$ can be taken as an $x\in \bC^n$, such that $x_i \neq 0$ for $i\in \bp$, and $x_i = 0$ for $i\not\in \bp$. The limit $\lim_{t\to 0} \lambda(t) \cdot x$ exists if and only if $\sum_{j=1}^k \iota_{ij} \lambda_j \geq 0$, for all $i\in \bp$. The Hilbert--Mumford criterion then implies that $\langle \lambda, \theta \rangle >0$ for all such $\lambda$. Therefore the lemma holds. 
\end{proof}

Consider all fixed points $\bp\in \fX$. The closure of each K\"ahler cone $\cK(\bp)$ is a rational polyhetral strictly convex cone in $\Lie_\bR (\sK^\vee)$. The codimension-one boundaries of such cones in $\Lie_\bR (\sK^\vee)$ then form the walls in the variation of GIT\footnote{The fan determined by such walls is the so-called \emph{secondary fan}. }. We have the following direct description of the chambers.

\begin{Lemma}
	Let $X \subset \fX$ be a generic GIT quotient. Then the K\"ahler cone $\cK (X)$ of $X$ is a chamber in the wall-and-chamber structure of the variation of GIT. More precisely, we have 
	$$
	\cK(X) = \bigcap_{\bp\in X} \cK(\bp), \qquad \Eff(X) = \bigcup_{\bp\in X} \Eff (\bp). 
	$$ 
\end{Lemma}

\begin{Example} \label{Bl(P^2)}

Consider the exact sequence	
	\begin{equation}\xymatrix{
		0 \ar[r] & \bZ^2 \ar[r]^-{\iota} & \bZ^4 \ar[r]^-\beta & \bZ^2 \ar[r] & 0,
	} 
	\end{equation}
where
$$
\iota = \begin{pmatrix}
1 & 1 \\
0 & 1 \\
1 & 0 \\
0 & 1 
\end{pmatrix}, \qquad \beta = \begin{pmatrix}
1 & 0  & -1 & -1 \\
0 & 1 & 0 & -1
\end{pmatrix}. 
$$
Denote $C_1 = \bR_+ \cdot (1,0) + \bR_+ \cdot  (1,1)$ \footnote{We denote $\bR_+ := (0, +\infty)$, and $\bR_- := (-\infty, 0)$. }, $C_2 = \bR_+ \cdot (1,1) + \bR_+ \cdot (0,1)$. There are 5 fixed points:
\begin{itemize}
\setlength{\parskip}{1ex}	
	
\item $\{1,2\}$, with K\"ahler cone $C_2$; 

\item $\{1, 3\}$, with K\"ahler cone $C_1$; 

\item $\{1, 4\}$, with K\"ahler cone $C_2$; 

\item $\{2,3\}$, with K\"ahler cone $\bR_+^2$, whose closure is $\overline C_1 \cup \overline C_2$; 

\item $\{3,4\}$, with K\"ahler cone $\bR_+^2$. 

\end{itemize}	
We can see that $C_1$ and $C_2$ are the two chambers in the variation of GIT. The GIT quotients for them are the following. 

\begin{itemize}
	\setlength{\parskip}{1ex}	
	
	\item Chamber $C_1$, $X = \bP^2$, containing fixed points $\{1,3\}$, $\{2,3\}$, $\{3,4\}$. 
	
	\item Chamber $C_2$, $X = \Bl_{\pt} \bP^2$, containing fixed points $\{1,3\}$, $\{1,4\}$, $\{2,3\}$, $\{2,4\}$. 

\end{itemize}
See Figure 1 below.	
\end{Example}

\begin{Example} \label{mirror-of-Bl(P^2)}
	Let's consider the dual exact sequence of that in Example \ref{Bl(P^2)}, i.e.
	$$
	\xymatrix{
		0 \ar[r] & \bZ^2 \ar[r]^-\iota & \bZ^4 \ar[r]^-\beta & \bZ^2 \ar[r] & 0,
	} 
	$$
	where
	$$
	\iota = \begin{pmatrix}
	1 & 0 \\
	0 & 1 \\
	-1 & 0 \\
	-1 & -1
	\end{pmatrix}, \qquad \beta = \begin{pmatrix}
	1 & 0 & 1 & 0 \\
	1 & 1 & 0 & 1
	\end{pmatrix}. 
	$$
Denote $C_1 = \bR_+^2$, $C_2 = \bR_- \times \bR_+$, $C_3 = \bR_+ \cdot (-1,0) + \bR_+ \cdot (-1,-1)$, $C_4 = \bR \cdot (-1, -1) + \bR_+ \cdot (1,0)$. There are also 5 fixed points:

\begin{itemize}
\setlength{\parskip}{1ex}		

\item $\{3,4\}$, with K\"ahler cone $C_3$; 

\item $\{2,4\}$, with K\"ahler cone $\bR_+ \cdot (0,1) + \bR_+ \cdot (-1,-1)$, whose closure is $\overline C_2 \cup \overline C_3$; 

\item $\{2,3\}$, with K\"ahler cone $C_2$; 

\item $\{1,4\}$, with K\"ahler cone $C_4$; 

\item $\{1,2\}$, with K\"ahler cone $C_1$. 

\end{itemize}
$C_i \ (1\leq 4)$ are the 4 chambers in the variation of GIT. The GIT quotients are the following. 

\begin{itemize}
\setlength{\parskip}{1ex}	

\item Chamber $C_1$, $X = \bC^2$, containing fixed point $\{1,2\}$. 

\item Chamber $C_2$, $X = \Bl_{\pt} \bC^2$, containing fixed points $\{2,3\}$, $\{2,4\}$. 

\item Chamber $C_3$, $X = \Bl_{\pt} \bC^2$, containing fixed points $\{2,4\}$, $\{3,4\}$. 

\item Chamber $C_4$, $X = \bC^2$, containing fixed point $\{1,4\}$. 
\end{itemize} 
See Figure 2 below.
\begin{align*}
\begin{tikzpicture}[line cap=round,line join=round,>=triangle 45,x=1.9cm,y=1.9cm]
\draw [color=cqcqcq,, xstep=0.5cm,ystep=0.5cm] (-4.594651950510948,-0.671428433701874);
\clip(-4.594651950510948,-0.771428433701874) rectangle (4.875368844973067,2.3217392055655575);
\fill[line width=1pt,color=yqyqyq,fill=yqyqyq,fill opacity=1] (-3,0) -- (-1,0) -- (-1,2) -- cycle;
\fill[line width=1pt,color=cqcqcq,fill=cqcqcq,fill opacity=1] (-3,0) -- (-1,2) -- (-3,2) -- cycle;
\fill[line width=1pt,color=aqaqaq,fill=aqaqaq,fill opacity=1] (1,1) -- (1,0) -- (2,1) -- cycle;
\fill[line width=1pt,color=wqwqwq,fill=wqwqwq,fill opacity=1] (2,1) -- (3,1) -- (3,2) -- (2,2) -- cycle;
\fill[line width=1pt,color=yqyqyq,fill=yqyqyq,fill opacity=1] (2,1) -- (2,2) -- (1,2) -- (1,1) -- cycle;
\fill[line width=1pt,color=cqcqcq,fill=cqcqcq,fill opacity=1] (2,1) -- (1,0) -- (3,0) -- (3,1) -- cycle;
\draw [->,line width=1.5pt] (-3,0) -- (-1,0);
\draw [->,line width=1.5pt] (-3,0) -- (-3,2);
\draw [->,line width=1.5pt] (-3,0) -- (-1,2);
\draw [->,line width=1.5pt] (2,1) -- (3,1);
\draw [->,line width=1.5pt] (2,1) -- (2,2);
\draw [->,line width=1.5pt] (2,1) -- (1,0);
\draw [->,line width=1.5pt] (2,1) -- (1,1);
\draw (-2.407709574462179,-0.16015520639458675) node[anchor=north west] {Figure 1};
\draw (1.6089034260868553,-0.1506596673861966) node[anchor=north west] {Figure 2};
\begin{scriptsize}
\draw[color=black] (-1.9709147800762319,-0.122469890772929768) node {$v_{1}$};
\draw[color=black] (-3.11037946108305,1.154976946267447) node {$v_{2}$};
\draw[color=black] (-1.4296690565979935,1.7152137477624652) node {$v_{3}$};
\draw[color=black] (2.871810114202746,0.7840572441164247) node {$v_1$};
\draw[color=black] (2.1691402275818743,1.8620977619806088) node {$v_2$};
\draw[color=black] (1.343028333851931,0.1813873574955546) node {$v_3$};
\draw[color=black] (1.1436220146757379,1.2209412583345685) node {$v_{4}$};
\draw[color=yqyqyq] (-1.6005887587490162,0.8131375419654022) node {$C_{1}$};
\draw[color=cqcqcq] (-2.274772028344717,1.4778252725527117) node {$C_{2}$};
\draw[color=aqaqaq] (1.3905060288938818,0.8131375419654022) node {$C_3$};
\draw[color=wqwqwq] (2.5584573269258706,1.6392494356953442) node {$C_1$};
\draw[color=yqyqyq] (1.5614257310449045,1.6392494356953442) node {$C_2$};
\draw[color=cqcqcq] (2.311573312707727,0.6422178398143797) node {$C_4$};
\end{scriptsize}
\end{tikzpicture}	
\end{align*}
\end{Example}

\begin{Remark}
In Section \ref{mirror}, we will see that Example \ref{Bl(P^2)} and \ref{mirror-of-Bl(P^2)} are mirror to each other. We can directly see that there is a natural bijection between the fixed points. However, this symmetry only exists if we include all $\sT$-fixed points of the stack $\fX$, i.e. of all possible GIT quotients. Also it might happen that different GIT quotients share a few common fixed points. 
\end{Remark}

\subsection{Attracting cone}

In this subsection, following Maulik--Okounkov \cite{MO}, we would like to associated to each fixed point $\bp\in \fX$ a cone in the space of equivariant parameters. 

Let $\bp \in \fX$ be a $\sT$-fixed point. First, let's choose $\theta \in \cK (\bp)$, and consider the GIT quotient $X = X_\theta$. By definition of the K\"ahler cone, $\bp$ is then a $\sT$-fixed point in $X$. 

The action of $\sT$ contains the kernel $\sK$ which acts trivially. Hence the actual torus acting on $X$ is the quotient torus $\sA = \sT / \sK \cong (\bC^*)^r$. Its Lie algebra $\Lie_\bR (\sA) \cong \bR^r$ can be identified with the space of cocharacters of $\sA$. For any cocharacter $\sigma: \bC^* \to \sA$, the 1-parameter subgroup $\bC^*_\sigma$ induced by it acts on $X$ and gives a Bialynicki--Birula stratification. Each stratum will be a union of $\sA$-orbits. When $\sigma$ is \emph{generic} (which we will always assume), the fixed loci of $\bC^*_\sigma$ will be the same as the fixed loci of $\sA$ itself. Therefore, the BB strata are the ``attracting sets", parameterized by the fixed points 
$$
\Attr_\bp^\theta := \{ q \in X_\theta \mid \lim_{t\to 0} \sigma (t) \cdot q = \bp \}. 
$$
There exists a unique $\bp$, such that $\Attr_\bp^\theta$ is the the largest  stratum, i.e. has the same dimension as $X_\theta$. In particular, this $\Attr_\bp^\theta$ contains points $q\in X_\theta$ whose representatives $x$ has all $x_i \neq 0$. In other words, this $\Attr_\bp^\theta$ contains the largest open $T$-orbit. We call such a $\bp$ the \emph{minimal} fixed point. 

\begin{Definition}
	The \emph{attracting cone} $\cA (\bp)$ is defined as
	$$
	\cA(\bp) := \{ \sigma \in \Lie_\bR (\sA)  \mid \bp \text{ is minimal under } \bC^*_\sigma \} . 
	$$
\end{Definition}

A priori, the definition depends on the choice of $\theta$. However, we can easily see that it actually does not depend on $\theta$. 

A cocharacter $\tilde\sigma \in \Lie_\bR (\sT)$ is called a \emph{lift} of $\sigma$ if it is a preimage of $\sigma$ along the map $\beta: \Lie_\bR(\sT) \to \Lie_\bR (\sA)$. 

\begin{Lemma} \label{attracting-cone}
	\begin{enumerate}[1)]
		
		\item There is a unique lift $\tilde\sigma = (\tilde\sigma_1, \cdots, \tilde\sigma_n)$ such that $\tilde\sigma_i = 0$,  for all $i \in \bp$. 
		
		\item $\bp$ is minimal, if and only if for the lift $\tilde\sigma$ in 1), we have $\tilde\sigma_i >0$, for all $i\not\in \bp$. In particular, the minimality of $\bp$ does not depend on the choice of the stability condition $\theta$. 
		
		\item $\cA(\bp)$ is the interior of the cone in $\Lie_\bR (\sA)$ generated by the $i$-th columns of $\beta$ with $i\not\in \bp$. 
		
	\end{enumerate}
\end{Lemma}

\begin{proof}
	For 1), all lifts $\tilde\sigma$ in $\Lie_\bR (\sT)$ form a $k$-dimensional affine linear subspace $\beta^{-1} (\sigma)$. Moreover, since $\sigma$ is generic, $\beta^{-1} (\sigma)$ intersects transversally with $x_i = 0$, $i\in \bp$. Therefore, they intersect at a unique point. 
	
	For 2), we know that $\bp$ is minimal if and only if for a generic point $q\in X_\theta$, $\lim_{t\to 0} \sigma(t) \cdot q = \bp$. Choose the lift $\tilde\sigma$ as in 1), and let $x$ be a representative of a generic point, which implies that $x_i \neq 0$ for all $i$. Then $\tilde\sigma(t) \cdot x$ is a representative of $\sigma(t) \cdot q$, and we see that
	$$
	( \tilde\sigma(t) \cdot x )_i = \left\{ \begin{aligned}
	& x_i, \qquad && i\in \bp \\
	& t^{\tilde\sigma_i} x_i, \qquad && i\not\in \bp
	\end{aligned} \right. 
	$$
	The limit as $t\to 0$ is $\bp$ if and only if $\tilde\sigma_i >0$ for all $i\not\in \bp$. 
	
	By 2), we see that the cone $\cA (\bp)$ is exactly the image under $\beta$ of the cone $\{0\} \times \bR_+^r$, where the $\{0\}$ is for the indices in $\bp$, and $\bR_+^r$ is for the indices not in $\bp$. The image is just the cone described in 3). 
\end{proof}

\subsection{Equivariant $K$-theory, K\"ahler and equivariant parameters} \label{parameters}

The $\sA$-equivariant $K$-theory ring of $\fX$ is
$$
K_\sA (\fX) = K_{\sA} ( [\bC^n / \sK ] ) \cong K_{\sK \times \sA } (\pt) . 
$$
However, the product $\sK \times \sA$ here is not canonical. There is no natural basis for the equivariant parameters of $\sA$; more precisely, for different fixed points $\bp$, there are different choices of decompositions $\sK \times \sA$ and different coordinates on $\sA$. 

A better way is to introduce the \emph{redundant parameters}. Recall the exact sequence
$$
\xymatrix{
	1 \ar[r] & \sK \ar[r] & \sT \ar[r] & \sA \ar[r] & 1 , 
} 
$$
where we view
$$
\sT = \Spec K_{\sT} (\pt) = \Spec \bC [a_1^{\pm 1}, \cdots, a_n^{\pm 1} ], \qquad \sA = \Spec K_{\sA} (\pt). 
$$
Here $a_i$ are the standard coordinates on $\sT$, or functions associated with the standard basis. We call them the \emph{redundant equivariant parameters}. We will then call functions on the quotient torus $\sA$ the \emph{effective equivariant parameters}. 

Similar phenomenon occurs for the K\"ahler parameters, and the dual exact sequence
$$
\xymatrix{
1 \ar[r] & \sA^\vee \ar[r] & \sT^\vee \ar[r] & \sK^\vee \ar[r] & 1.  
}
$$
We write
$$
\sT^\vee = \Spec K_{\sT^\vee} (\pt) = \Spec \bC [z_1^{\pm 1}, \cdots, z_n^{\pm 1} ], 
$$
where $z_i$'s are the \emph{redundant K\"ahler parameters}. Functions on $\sK^\vee$ are called the \emph{effective K\"ahler parameters}. We denote the standard coordinates on $\sK^\vee$ by 
$$
Q_j, \qquad 1\leq j\leq k. 
$$
The relationship between $z_i$'s and $Q_j$'s is
$$
Q_j = \prod_{i=1}^n z_i^{\iota_{ij}}, \qquad 1\leq j\leq k. 
$$

\begin{Remark}
The effective parameters are what usually appear in the literature, which records the degrees of curves. The redundant ones, introduced by the third author in a previous work \cite{SZ}, is a set of globally chosen coordinates on the K\"ahler and equivariant tori, which would make the presentation of the $I$-functions much more convenient. The mirror map also looks more concise in terms of the redundant parameters. 
\end{Remark}

Now we apply the base change $\sT \to \sA$ to the $\sA$-equivariant theory and consider the $\sT$-equivariant $K$-theory of $\fX$. The ring admits a global presentation
\begin{equation} \label{pres-u}
K_{\sT} (\fX) = \bC [ a_1^{\pm 1} , \cdots, a_n^{\pm 1} , u_1^{\pm 1} , \cdots, u_n^{\pm 1} ] / \langle \prod_{j=1}^n \left( \frac{u_j}{a_j} \right)^{\beta_{ij}} = 1,  \ 1 \leq i\leq r \rangle, 
\end{equation}
where the $u_i$'s are the characters associated with the 1-dimensional representations given by the standard basis in $\bC^n$. 

It is easy to see that the monomials $\prod_{j=1}^n a_j^{\beta_{ij}}$ appearing in the relations are functions on the quotient torus $\sA$; hence the ring is indeed a base change from $\sA$. The picture is the following Cartesian diagram
$$
\xymatrix{
	\Spec K_{\sT} (\fX)  \ar[d] \ar[r] & \Spec K_\sA (\fX) \ar[d] \\
	\sT \ar[r] & \sA. 
}
$$

The $K$-theory spectrum $\Spec K_\sT (\fX)$ is a fibration of $k$-torus over $\sT$, and the ring $K_\sT (\fX)$ is of infinite dimensions over $K_\sT (\pt)$. Same holds for $\sA$. An alternative way to look at this is as follows. Let $s_j$ be the $T$-character associated with the $j$-th standard basis vector in $\bC^k$. We have the relation
\begin{equation} \label{u-s}
u_i = a_i \prod_{j=1}^k s_j^{\iota_{ij}}, \qquad 1\leq i\leq n, 
\end{equation}
and hence we have an alternative presentation of the $K$-theory ring
\begin{equation} \label{pres-s}
K_{\sT} (\fX) = \bC [ a_1^{\pm 1} , \cdots, a_n^{\pm 1} , s_1^{\pm 1} , \cdots, s_k^{\pm 1} ]. 
\end{equation}
We see that $s_j$'s form a set of global coordinates of the above fibration. In computations of the later sections, we will always consider $a_i$ and $s_i$ as independent parameters, and $u_i$'s as functions in terms of $a_i$ and $s_i$'s. 

For any (generic) GIT quotient $X \subset \fX$, there is a Kirwan surjection \cite{HL, BH} $K_\sT (\fX) \twoheadrightarrow K_\sT (X)$, under which $u_i$ and $s_i$ map to the corresponding tautological line bundles on $X$, equvariant and nonequivariant respectively. The ring $K_\sT (X)$ then admits similar presentations as (\ref{pres-u}) and (\ref{pres-s}), with some extra relations described by the following lemma. 

\begin{Lemma}
	Let $\bp$ be a fixed point of $\fX$. 
	\begin{enumerate}[1)]
		
		\item $\Spec K_\sT (\bp)$ is the section of the fibration $\Spec K_\sT (\fX) \to \sT$ defined by $u_i = 1$, for all $i\in \bp$. 
		
		\item For any generic GIT quotient $X$, its  $K$-theory ring is
		$$
		\Spec K_\sT (X) = \bigcup_{\bp \in X} \Spec K_\sT (\bp), 
		$$
		where the irreducible components intersect transversally. 
		
	\end{enumerate}
\end{Lemma}

We see that the restriction of the line bundle $u_i$ to a fixed point $\bp$ is $u_i |_\bp = 1$ for $i\in \bp$, and 
$$
\{ u_i |_\bp \mid i\not\in \bp \}
$$
is the unique solution to the system of linear equations
\begin{equation} \label{solution}
\left\{  \begin{aligned} 
& \prod_{j=1}^n \left( \frac{u_j}{a_j} \right)^{\beta_{ij}} = 1 , \qquad && 1\leq i\leq r \\
& u_l = 1, \qquad && l\in \bp . 
\end{aligned} \right. 
\end{equation}

In terms of $s_j$'s, the fixed point $\bp$ is given by the equations
\begin{equation}
\prod_{j=1}^k s_j^{\iota_{lj}} = a_l^{-1}, \qquad l\in \bp. 
\end{equation}

In particular, $u_i |_\bp$'s are functions on $\sA$, i.e. only in terms of effective equivaraint parameters; however $s_j |_\bp$'s involve all redundant equivariant parameters.

\vspace{1cm}

\section{$K$-theoretic $I$-function for toric stacks} \label{sec-3}
The quantum $K$-theory was introduced by Givental \cite{Giv-WDVV} and Y.-P. Lee \cite{Lee} decades ago. Recently, Givental shows that $q$-hypergeometric solutions represent $K$-theoretic Gromov--Witten invariants in the toric case \cite{Giv-perm-V}. Ruan--Zhang \cite{RZ} introduce the level structures and there is a serendipitous discovery that some special toric spaces with certain level structures result in Mock theta functions. Nevertheless, beyond the toric case, much less is known.

Let $X := X_\theta$ be a GIT quotient $V/\!/_{\theta} G$ where $V$ is a vector space and $G$ is a connected complex reductive group. The theory of the moduli space of quasimaps to GIT is established in \cite{CKM} where Ciocan-Fontanine, Kim and Maulik define the cohomological big $I$-function. The first two authors prove the wall-crossing formula which relates big $I$-function and Givental's big $J$-function of $X$ in their following paper \cite{CK}. The $K$-theoretic stable quasimaps invariants are defined by Tseng--You in \cite{TY}.

Let ${\mathcal{Q}}^{\epsilon}_{g,n}(X, d)$ be the moduli stack of $\epsilon$-stable quasimaps \cite{CKM} parametrizing quasimaps $f=(C,p_1, \cdots ,p_n,\mathcal{P},s)$ where $C$ is an $n$-pointed nodal curve of genus $g$, $\mathcal{P}$ is a principal $G$-bundle over $C$, $s$ is a section and $d \in \mathrm{Hom}(\mathrm{Pic}^G(V), \bZ)$. There are natural maps:
\begin{align*}
\ev_{i}: {\mathcal{Q}}^{\epsilon}_{g,n}(X, d) \rightarrow X, \qquad i=1, \cdots, n, 
\end{align*}
given by evaluation at the $i$-th marked point. There are line bundles
\begin{align*}
	\mathbb{L}_{i} \rightarrow  {\mathcal{Q}}^{\epsilon}_{g,n}(X, d) , \qquad i=1, \cdots, n, 
\end{align*}
called universal cotangent line bundles. The fiber of $\mathbb{L}_i$ over a point $(C,p_1, \cdots ,p_n,\mathcal{P},s)$ is the cotangent line to $C$ at the point $p_i$.

The permutation-equivariant $K$-theoretic quasimap invariants with level structures \cite{RZ} are holomorphic characteristics over ${\mathcal{Q}}^{\epsilon}_{g,n}(X,d)$ of the sheaves:
\begin{align}
\left\langle \mathbf{t}(\mathbb{L}_1), \cdots, \mathbf{t}(\mathbb{L}_n) \right\rangle_{g, n, d}^{R,l,S_n,\epsilon} :=\pi_*  \Big( {\mathcal{Q}}^{\epsilon}_{g,n}(X, d) ; \,  \mathcal{O}_{g, n, d}^\vir \otimes \prod_{m,i} \mathbb{L}_{i}^{k}  t_{k,i} \mathrm{ev}_{i}^{*}\left(\phi_{i}\right) \otimes \mathcal{D}^{R,l} \Big)  \label{quasimap-invariants}
\end{align}
where $\mathcal{O}^\vir_{g,n,d}$ is the virtual structure sheaf \cite{Lee}, and  $\mathbf{t}(q)$ is a Laurent polynomial in $q$ defined as follows 
\begin{align*}
	\mathbf{t}(q)=\sum_{m \in \mathbb{Z}} t_{m} q^{m}, \qquad t_{m}=\sum_{\alpha} t_{m, \alpha} w_{\alpha}.
\end{align*}
Moreover, $\pi_*$ is the $K$-theoretic pushforward along the projection 
\begin{align*}
 \pi_* : \left[ {\mathcal{Q}}^{\epsilon}_{g,n}(X, d )/S_n \right] \rightarrow \pt, 
\end{align*}
$\{\phi_\alpha \}$ is a basis of $K^0(X_\theta)\otimes \bQ$, and $t_{k,\alpha}$ are formal variables. The last term in (\ref{quasimap-invariants})  is the level $l$ determinant line bundle over $\mathcal{Q}^{\epsilon}_{g,n}(X_\theta, \beta)$, defined as
\begin{align*}
	\mathcal{D}^{R,l} :=\left( \mathrm{det}R^{\bullet}\pi_*(\mathcal{P} \times _G R ) \right)^{-l} , 
\end{align*}
where the bundle $\mathcal{P} \times _G R $ is the pullback of the vector bundle $[V \times R / G]  \rightarrow [ V/G ]  $ along the evaluation map to the quotient stack $[V/G]$.

Similarly, we can define the moduli space for graph space quasimaps $ \mathcal { QG }^{\epsilon} _ { 0 , n } ( X , d ) $,  which parametrizes quasimaps with \emph{parametrized} domain component $\mathbb{P}^1$. As a result, there is a natural $\mathbb{C}^*$-action, coming from the $\bC^*$-action that scales the parametrized domain component. Denote by $F_{0, d}$ the special fixed loci in $  \mathcal { QG }^{\epsilon} _ { 0 , n } ( X , d )^{\mathbb{C}^*}$, i.e., the open substack consisting of quasimaps $f$ such that $\infty \in \bP^1$ is not a base point, and denote by $q$ the $\bC ^*$-character of the cotangent bundle at $0 :=[1,0] $ of $\mathbb{P}^1$ (sometimes we denote by $\bC_q^*$, the same $\bC^*$ to emphasize the character). For details, see \cite{CKM}.
\begin{Definition} {\cite{RZ}}
The permutation-equivariant $K$-theoretic $\mathcal{J}^{R,l,\epsilon}$-function of $V /\!/ {G}$ with level $l$ is defined as
\begin{align*}
\mathcal { J } _ { S _ { \infty } } ^ { R , l , \epsilon } ( \mathbf { t } ( q ) , Q ) 
 &:= \sum_{k \geq 0, \, d \in {\operatorname { Eff } ( V , G , \theta )} } Q^d (\ev_{\bullet})_{*} \left[ \operatorname{Res}_{F_{0, d}}( \mathcal{QG} _ { 0 , n } ^ { \epsilon } ( V /\!/ G , d )_{0})^{\mathrm{vir}} \otimes \mathcal { D } ^ { R , l } \otimes _ { i = 1 } ^ { n } \mathbf { t } ( \mathbb{L} _ { i } ) \right]^{S_n} \\
 &:= 1 + \frac { \mathbf { t } ( q ) } { 1 - q } +\sum _ { a } \sum _ { d \neq 0 } Q ^d \chi \left( F_ { 0 , d } , \, \mathcal { O } _ { F_ { 0 , d} } ^ { \mathrm { vir } } \otimes \ev_{\bullet} ^ { * } ( \phi _ { a } ) \otimes \left( \frac { \operatorname { t r } _ { \mathbb { C } ^ { * } } \mathcal { D } ^ { R , l } } { \lambda_{-1}^{\mathbb{C}^*}  N _ { F_ { 0 , d } } ^ { \vee } } \right) \right) \phi ^ { a } \\
 & + \sum_a \sum _ { n \geq 1 \, \text{or} \,  d(L_\theta) \geq 1 / \epsilon \atop ( n , d ) \neq ( 1,0 ) } Q ^d \left\langle \frac { \phi _ { a } } { ( 1 - q ) ( 1 - q \mathbb{L}_{n+1} ) } , \mathbf { t } ( \mathbb{L}_1 ) , \cdots , \mathbf { t } ( \mathbb{L}_n ) \right\rangle _ { 0 , n + 1 , d } ^ { R , l , \epsilon , S _ { n } } \phi ^ { a } , 
\end{align*}
where $\ev_\bullet$ is the evaluation map at the point $\infty\in \bP^1$, $\{ \phi_\alpha \}$ is a basis of $K^0(X)$ and $\{ \phi^\alpha \}$ is the dual basis with respect to twisted pairing  $( \ \ , \ \ )^{R,l}$,  i.e., 
\begin{align*}
	(u,v)^{R,l}:=\chi\left( X,u \otimes v \otimes \mathrm{det}^{-l}(V^{ss} \times_{G} R )  \right) . 
\end{align*}
\end{Definition}
\begin{Definition}{\cite{RZ}}
When $\epsilon$ is small enough (denoted by $\epsilon=0^{+}$), we call $\mathcal{J}^{R,l,0^{+}}(0)$ the small $I$-function of level $l$, i.e,
\begin{align*}
{I}^{R,l}(q;Q):= \mathcal { J } _ { S _ { \infty } } ^ { R , l , 0^{+} } ( 0 , Q ) =  1 + \sum _ { d \geq 0 } Q^d (\ev_{\bullet})_{*} \left(  \mathcal { O } _ { F_ { 0 , d } } ^ { \mathrm { vir } } \otimes  \left( \frac { \operatorname { t r } _ { \mathbb { C } ^ { * } } \mathcal { D } ^ { R , l } } { \lambda_{-1}^{\mathbb{C}^*}  N _ { F_ { 0 , d } } ^ { \vee } } \right) \right) \cdot \mathrm{det}^l(V^{ss} \times_{G} R ) . 
\end{align*}
Note that here we take $(g,n) = (0,0)$. 
\end{Definition}
In the rest of this section, we will compute the explicit formula for the contribution of fixed points to the $I$-function.

\subsection{Quasimaps to fixed points and their $I$-functions}
From now on, let $V = \bC^n$, and let $G = \sK \cong (\bC^*)^k$ act on $V$ by the charge matrix $\iota$. Consider the GIT quotient $X_\theta = \bC^n/\!/_\theta \sK $ with respect to a character $ \theta \in \Lie_{\bR}(\sK^\vee) $, there is a natural $\sT$--action on $X_\theta$, which induces an action on $\mathcal{Q}_{0,0}(X_\theta,d)$, then
\begin{align*}
\mathcal{Q}_{0,0}(X_\theta,d)^{\sT} = \bigsqcup_{\bp \in X_{\theta}^{\sT}} \mathcal{Q}_{0,0}(\bp,d)	, 
\end{align*}
where $\mathcal{Q}_{0,0} (\bp, d)$ be the moduli space of quasimaps from $\bP^1$ to $\bp$ of degree $d$.

Recall that $u_i$ and $s_j$, for $1\leq i\leq n$, $1\leq j\leq k$, are the characters associated with the basis of $\bC^n$ and $\bC^k$ respectively. Let $(-)|_{X_\theta}$ be the restriction map $K_\sT(\fX) \to K_\sT (X_\theta)$ induced by the open inclusion $X_\theta \hookrightarrow \fX$. Denote by 
$$
U_i := u_i |_{X_\theta},  \qquad L_j := s_j |_{X_\theta} 
$$
the tautological line bundles on $X_\theta$ associated with the standard characters. By abuse of notations, we often use the same letters $U_i$ and $L_j$'s for line bundles pulled back to $\bP^1$, i.e. $f^*u_i$ and $f^* s_j$. 

An alternatively description of a quasimap is the datum consisting of $k$ line bundles $\{ L_j \}_{j=1}^{k}$ (which are identified with $f^* s_j$'s), and a certain section of the associated vector bundle $\bigoplus_{i=1}^n U_i$ (which is identified with $\bigoplus_{i=1}^n f^* u_i$). 
\begin{Definition} \label{total-stack-I}
The \emph{modified $I$-function} is defined as 
$$
\widetilde I^{R,l} (q,Q):= e^{-\sum_{i=1}^n\frac{\ln z_i \ln U_i}{\ln q}} \cdot \frac{ {\lambda_{-1} (T^* X_{\theta})} }{(q \cdot T X_\theta )_\infty} \cdot I^{R,l}(q,Q), 
$$
where $(-)_\infty$ and $\lambda_{-1} (-)$ are characteristic classes defined as
$$
(E)_\infty := \prod_i (\cL_i)_\infty, \quad \lambda_{-1} (E)=\prod_{i}(1- \cL_i)
$$
if a vector bundle $E$ splits into line bundles $E = \bigoplus_i \cL_i$. Recall that $z_i$'s are the redundant K\"ahler parameters. 
\end{Definition}

\begin{Remark}
The exponential prefactor $e^{-\sum_{i=1}^n\frac{\ln z_i \ln U_i}{\ln q}}$ is a $q$-analogue of the exponential factor $e^{\sum_i \frac{t_i H_i}{z}}$ as in the cohomological $I$-functions. The factor $\dfrac{1}{(q \cdot TX_\theta )_\infty}$ are $q$-analogues of the Gamma class, which lies in a  \emph{completion} of the $K$-group.  
\end{Remark}

The \emph{degree} of a quasimap is defined as $d = (d_1, \cdots, d_k)\in \bZ^k$, where $d_j = \deg L_j $, for $1\leq j\leq k$. One can also describe it by $D = (D_1, \cdots, D_n) \in \bZ^n$, where $D_i = \deg U_j$, for $1\leq i\leq n$. The relation between them, by (\ref{u-s}), is
$$
D_i = \sum_{j=1}^k \iota_{ij} d_j, \qquad \text{or} \qquad D = \iota d. 
$$
A point in $\bP^1$ is called a base point if it is not mapped to $\bp$ under $f$. 

Let $\mathcal{Q}_{0,0} (\bp, d)^\circ$ be the open substack consisting of quasimaps $f$ such that $\infty \in \bP^1$ is not a base point. We have the following diagram
$$
\xymatrix{
\mathcal{Q}_{0,0} (\bp, d)^\circ \ar@{^{(}->}[r] & \mathcal{Q}_{0,0} (\bp, d) \ar@{^{(}->}[r] & \mathcal{Q}_{0,0} (X_\theta, d), 
}
$$
where the second inclusion is a closed embedding. There is a perfect obstruction theory on $\mathcal{Q}_{0,0} (X_\theta, d)^\circ $, (whose dual is) given by 
$$
R\pi_* f^* T X_\theta, 
$$
where $\pi: \bP^1 \times \Hom (\bP^1, X_\theta) \to \Hom(\bP^1, X_\theta)$ is the universal curve, and $ f: \bP^1 \times \Hom (\bP^1, X_\theta) \to X_\theta $ is the universal morphism. Let $T_\vir (\bp)$ be the pull-back of $R\pi_* f^* T X_\theta$ to $\mathcal{Q}_{0,0} (\bp, d)^\circ$, which is (the dual of) a perfect obstruction theory on $\mathcal{Q}_{0,0} (\bp, d)^\circ$. 

Let $\cO_\vir (\bp)$ be the virtual structure sheaf \cite{Lee} associated with the obstruction theory $T_\vir (\bp)$. There is an evaluation map $\ev_\infty: \mathcal{Q}_{0,0} (\bp, d)^\circ \to X_\theta$, which is not proper. From the following commutative diagram
$$
\xymatrix{
\mathcal{Q}_{0,0} (\bp, d)^\circ \ar[d]^-{\ev_\infty} \ar@{^{(}->}[r]  & \mathcal{Q}_{0,0} (X_\theta, d)^\circ  \ar[d]^-{\ev_\infty}  \\
 \bp  \ar@{^{(}->}[r]     &  X_\theta , 
}	
$$
by using Atiyah-Bott localization formula, we obtain
\begin{align*}
I^{R,l}(q,Q) =  \sum_{d\in \Eff (X_\theta)} \sum_{\bp \in X_{\theta}^{\sT}} Q^d (\ev_\infty)_* \left( \cO_\vir (\bp) \otimes  \operatorname { t r } _ { \mathbb { C } ^ { * } } \mathcal { D } ^ { R , l }  \right)\cdot \frac{ \mathrm{det}^l(\bC^{n,s} \times_{\sK} R )}{\lambda^{\sT}_{-1} (T_\bp^* X_{\theta} )} . 
\end{align*}
It is a formal power series lying in 
$$
K_{\sT \times \bC^*_q} (\bp)_\loc \llbracket Q^{\Eff(\bp)} \rrbracket , 
$$
where ``loc" means to tensor with the fractional field, $Q_j$ with $1\leq j\leq k$ are the effective K\"ahler parameters, and $Q^d = \prod_{j=1}^k Q_j^{d_j}$.

\begin{Remark} \label{effective-Kahler}
By $D = \iota d$, the relation between effective and redundant parameters is $Q_j = \prod_{i=1}^n z_i^{\iota_{ij}}$. It then follows that $Q^d = z^D = \prod_{i=1}^n z_i^{D_i}$. Later we will use redundant parameters $z_i$ more often. 
\end{Remark}

We claim the contribution from each fixed point $\bp$ to the $I$-function is 
\begin{align} \label{I-of-fixed-point}
I^{R,l}(\bp,\theta)	:=\sum_{d\in \Eff (\bp)} Q^d (\ev_\infty)_* \left( \cO_\vir (\bp) \otimes  \operatorname { t r } _ { \mathbb { C } ^ { * } } \mathcal { D } ^ { R , l }  \right)\cdot \frac{ \mathrm{det}^l(\bC^{n,s} \times_{\sK} R )}{\lambda^{\sT}_{-1}(T_\bp^* X_{\theta} )} . 
\end{align}
Note that the summation is over $\Eff(\bp)$, which is a subcone of $\Eff(X_\theta)$. This is based on the following observation. 

\begin{Lemma} \label{key-lemma}
	Let $f$ be a quasimap from $\bP^1$ to $\bp$. Then
	\begin{enumerate}[1)]
		
		\item $D_i \geq 0$, for all $i\in \bp$; 
		
		\item the vector $d = (d_1, \cdots, d_k)$ lies in the effective cone $\Eff(\bp)$. 
		
	\end{enumerate}
\end{Lemma}

\begin{proof}
	Since $f$ generically maps into $\bp$, we know that for $i\in \bp$, the section of the line bundle $U_i$ defined by $f$ is generically nonzero. Therefore $D_i \geq 0$ for all $i\in \bp$. In other words, $\sum_{j=1}^k \iota_{ij} d_j \geq 0$ for all $i\in \bp$, which means that $d$ lies in the dual of $\cK (\bp)$. 
\end{proof}

\begin{Corollary} \label{independence-I-bp}
The contribution of the $I$-function $I^{R, l} (\bp, \theta)$ is independent of the choice of the stability condition $\theta$. 
\end{Corollary}

\begin{proof}
Every step in the localization computation can be performed on the quotient stack $\fX$, instead of the GIT quotient $X_\theta$. More precisely, one has the embedding
$$
\xymatrix{
	\mathcal{Q}_{0,0} (\bp, d)^\circ \ar@{^{(}->}[r] & \mathcal{Q}_{0,0} (\bp, d) \ar@{^{(}->}[r] & \mathcal{Q}_{0,0} (X_\theta, d) \ar@{^{(}->}[r] & \Hom (\bP^1, \fX), 
}
$$
where $\Hom (\bP^1, \fX)$ is the Artin stack parametrizing representable morphisms from $\bP^1$ to $\fX$. The obstruction theory on $\mathcal{Q}_{0,0} (X_\theta, d)$ is the restriction from an obstruction theory 
$$
T_\vir := R\pi_* f^* T \fX,
$$
where $T \fX$ is the tangent complex of the stack $\fX$, and $\pi$, $f$ are similar as above. One can then define the obstruction theory $T_\vir (\bp)$ alternatively, as the restriction of $T_\vir$ to $\mathcal{Q}_{0,0} (\bp, d)^\circ$. Replacing $T_p^* X_\theta$ by the same space $T_\bp \fX$, we see that every factor in the formula \ref{I-of-fixed-point} can be defined directly from the embedding $\bp \hookrightarrow \fX$, without a choice of $X_\theta$. Recalling that the effective cone $\Eff(\bp)$ also does not depend on $\theta$, the corollary follows. 
\end{proof}

\begin{Definition}
According to Corollary \ref{independence-I-bp}, we denote from now on 
$$
I^{R, l} (\bp) := I^{R, l} (\bp, \theta), 
$$
for any $\theta$ such that $\bp \in X_\theta$. The $I$-function with level structure for the toric stack is then defined as 
$$
I^{R,l} (\fX) :=  \sum_{\bp \in \fX^{\sT} } I^{R,l}({\bp}) \quad \in \quad \bigoplus_{\bp\in \fX^\sT} K_{\sT \times \bC^*_q} (\bp)_\loc \llbracket Q^{\Eff(\bp)} \rrbracket .		
$$
Similarly, the modified $I$-function with levels structure is defined as as 
$$
\widetilde{I}^{R,l} (\fX) :=  \sum_{\bp \in \fX^{\sT} } \widetilde{I}^{R,l}({\bp}).		
$$
\end{Definition}

\subsection{Explicit formula for $I$-functions}

Let $\bp\in \fX_\theta \subset \fX$ be a $\sT$-fixed point. Recall that the restriction of the line bundle $U_i |_\bp$ is the character $u_i |_\bp$, given as in (\ref{solution}). 

The virtual tangent bundle, restricted to a fixed quasimap $f$, is
\ben
T^\vir |_f &=& H^\bullet (\bP^1, U_1 |_\bp \oplus \cdots \oplus U_n |_\bp ) - H^\bullet (\bP^1, \cO^{\oplus k} ) \\
&=& \sum_{i=1}^n H^\bullet (\bP^1, U_i |_\bp \otimes \cO (D_i) ) - k \\
&=& \sum_{i\in \bp} ( H^\bullet (\bP^1, \cO (D_i) ) -1) + \sum_{i\not\in \bp} H^\bullet (\bP^1, U_i |_\bp \otimes \cO (D_i) ) . 
\een
Recall that we have
\ben
H^\bullet (\bP^1, \cO (d)) &=& \left\{ \begin{aligned}
	& 1 + q^{-1} + \cdots + q^{-d} , && \qquad d\geq 0 \\
	& - q - \cdots - q^{-d-1}, && \qquad d<0. 
\end{aligned} \right. \\
&=& \sum_{l=0}^\infty q^{-l} - \sum_{l=d+1}^\infty q^{-l} . 
\een
Since \footnote{We write $(x)_\infty := (x; q)_\infty$. When we need symboles such as $(x; q^{-1})_\infty$ or $(x; q^2)_\infty$, we do not omit the $q^{-1}$ or $q^2$. }
$$
\lambda_{-1} [U \cdot (H^\bullet (\bP^1, \cO(d)) -1) ]^\vee = \frac{\prod_{l=-\infty}^{d} (1-U^{-1} q^l)}{\prod_{l=-\infty}^{0} (1-U^{-1} q^l) }  =: (q U^{-1})_d , 
$$
the $I$-function is
\ben
I (\bp) &:=& \sum_{d \in \Eff(\bp)} \frac{1}{\lambda_{-1} (T^\vir |_f )^\vee } \\
&=& \frac{1}{\prod_{i \notin \bp}(1-U^{-1}_i|_{\bp})} \cdot \sum_{d\in \Eff (\bp)} \frac{Q^d}{\prod_{i=1}^n (q U_i^{-1} |_\bp )_{D_i} } .  
\een

\begin{Remark} \label{sum-over-lattice}
The summation over $d\in \Eff (\bp)$ here can actually be replaced with summation over the entire lattice $d\in \bZ^k$. The reason is that one always have
$$
\frac{1}{(qU_i^{-1}|_{\bp})_d}=\frac{1}{(q)_d} = 0, \qquad \text{if} \ d<0 \quad \text{and} \quad i \in \bp. 
$$
\end{Remark}

\begin{Remark}
The $I$-function is invariant under the $GL(k, \bZ)$-action on $\bZ^k$ or the $GL(d, \bZ)$-action on $\bZ^d$, i.e. the matrices $\iota$ (resp. $\beta$) we start with can be replaced up to a right multiplication by a matrix in $GL(k, \bZ)$ (resp. left multiplication by a matrix in $GL(d, \bZ)$). Under such a change of basis, the redundant parameters $z_i$ are unchanged, while the effective parameters $Q_j$ will be changed accordingly. 
\end{Remark}

\subsection{Effective level and modified $I$-function}
In Section \ref{sec-4}, we will use the following special representation in the mirror symmetry construction. Let
\begin{align}
R = \operatorname{Hom}(\bC,\bC^n) \simeq \bC^n, 
\end{align}
and the $\sK=(\bC^*)^k$-action be given by the matrix $\iota$ in the exact sequence (\ref{knd}). 

We choose the level $l = 1$. Then
\begin{align*}
\tr_{\bC^*} \cD^{\bC^n,1} &= tr_{\bC^*} \left( \operatorname{det}^{-1}R^{\bullet} \left( \oplus_{i=1}^n U_i \otimes \cO_{\bP^1}(D_i)  \right)   	\right) \\
&= \bigotimes_{i=1}^n \left(U_i^{-D_i+1}q^{D_i(D_i+1)/2}  \right) , 
\end{align*}
and 
\begin{align*}
	\operatorname{det}\left( \bC^{n,s} \times_{\sK} R  \right) = \bigotimes_{i=1}^{n} U_i . 
\end{align*}
Thus
\begin{align} \label{I-function-with-special-level}
I^{\bC^n,1}(\bp)	&= \frac{1}{\prod_{i \notin \bp}(1-U^{-1}_i|_{\bp}) } \cdot \sum_{d\in \Eff (\bp)} \frac{\prod_{i=1}^{n}\left( U_i^{-1}|_{\bp}q^{D_i(D_i+1)/2}   \right)  Q^d}{ \prod_{i=1}^n (q U_i^{-1} |_\bp  )_{D_i} }	  \\
&= \prod_{i \notin \bp}(1-U^{-1}_i|_{\bp}) \cdot \sum_{d\in \Eff (\bp)} \frac{\prod_{i=1}^{n}\left( U_i^{-1}|_{\bp}q^{D^2_i/2} \cdot    (q^{1/2}z_i)^{D_i}  \right) }{\prod_{i=1}^n (q U_i^{-1} |_\bp ; q )_{D_i} }  . 
\end{align}
Here we use the relation between effective parameters $Q_i$'s and redundant parameters $z_i$'s, see Remark \ref{effective-Kahler}.
\begin{Remark}
If we write out the expression of 
\begin{align*}
\frac{1}{2}\sum_{j=1}^n D_j^2 &=\frac{1}{2} \sum_{j=1}^n \Big(\sum_{i=1}^{k} \iota_{ij}\cdot d_i \Big)^2 \\
&= \frac{1}{2} \sum_{j=1}^n  \sum_{a,b=1}^k \iota_{aj} \iota_{bj} \cdot d_a d_b 
\end{align*}
{the term $\frac{1}{2} \sum_{a,b=1}^k \iota_{aj}\iota_{bj} $ appears in physics literatures as effective Chern-Simons term, thus we define this special level as follows:}
\end{Remark}

\begin{Definition} \label{effective-level}
We call the number $\frac{1}{2}\sum_{a,b=1}^k \iota_{a,j}\iota_{bj}, j=1,\cdots,n$ the \emph{effective levels}, and we call (\ref{I-function-with-special-level}) the $I$-function with \emph{effective levels}, denoted by $I^{\eff}(\bp)$. 
\end{Definition}

From the above computations, the \emph{modified $I$-function with effective level}  $\widetilde{I}^{\eff}(\bp)$ is 
$$
\widetilde{I}^{\eff} (\bp) := e^{-\sum_{i\not\in \bp} \frac{\ln z_i \ln U_i |_\bp }{\ln q} } \cdot  \prod_{i\not\in \bp} \frac{1-U_i^{-1}|_{\bp}}{ (U_i |_\bp )_\infty } \cdot  I^{\eff}  (\bp). 
$$
Note that the $Q$-coefficients of $\widetilde I (\bp)$ no longer lie in $K_{\sT \times \bC_q^*} (\bp)_\loc$.

\vspace{1cm}

\section{3d $\cN = 2$ mirror symmetry} \label{mirror} \label{sec-4}
Recall the toric stack $\fX$ is defined according to the following short exact sequence 
\begin{equation} \label{exact-sequnce-1} 
\xymatrix{
	0 \ar[r] & \bZ^k \ar[r]^-\iota & \bZ^n \ar[r]^-\beta & \bZ^d \ar[r] & 0. 
} 
\end{equation}
Inspired by \cite{DT, AHKT}, we consider the dual short exact sequence, i.e., the Gale dual to (\ref{exact-sequnce-1}):
\begin{equation} \label{exact-sequnce-2} 
\xymatrix{
	0 \ar[r] & (\bZ^d)^\vee \ar[r]^-{\iota^!} & (\bZ^n)^\vee \ar[r]^-{\beta^!} & (\bZ^k)^\vee \ar[r] & 0,
} 
\end{equation}
where 
$$
\iota^! := \beta^T, \qquad \beta^! := \iota^T. 
$$

\begin{Definition} \label{defn-mirror}
	
Let $\sA := \sT / \sK \cong (\bC^*)^d$ be the quotient torus. 

\begin{itemize}
		
		\setlength{\parskip}{1ex}
				
\item The \emph{mirror toric stack} to $\fX$ is defined as
$$
\fX^! := [\bC^n / \sA^\vee], 
$$
the toric quotient stack associated with the short exact sequence (\ref{exact-sequnce-2}), where the action by $\sA^\vee$ is defined by $\iota^!$. 

\item There is a natural bijection between fixed points of a mirror pair. Given a fixed point $\bp$ of $\fX$, the \emph{mirror fixed point} is defined as the complement
$$
\bp^! = \{1, \cdots, n \} \backslash \bp \quad \in \quad (\fX^!)^\sT . 
$$
	
\end{itemize}

\end{Definition}

It is easy to check that the columns of $\beta$, i.e. the rows of $\iota^!$, corresponding to $\bp^!$, are linearly independent. 

\begin{Lemma}
	The K\"ahler and attracting cones of $\bp^!$ are
	$$
	\cK (\bp^!) = \cA(\bp) , \qquad \cA (\bp^!) = \cK (\bp). 
	$$
\end{Lemma}

\begin{proof}
This follows from the combinatorial descriptions (Lemma \ref{key-lemma} and Lemma \ref{attracting-cone} 3)) of the K\"ahler and attracting cones of fixed points. 
\end{proof}

\begin{Definition} \label{mirror-map}
The \emph{mirror map} is defined as an isomorphism of tori $\tau: \sT \times \sT^\vee \times \bC_q^*  \cong \sT^\vee \times \sT \times \bC_q^*$, 
\begin{equation} \label{tau}
\tau (z_i^!) = a_i, \qquad \tau (a_i^!) = z_i, \qquad \tau (q) = q^{-1}. 
\end{equation}
\end{Definition}

We would like to apply the mirror map $\tau$ to the modified $I$-functions $\widetilde I(\bp)$. However, the map $q\mapsto q^{-1}$ only makes sense for rational functions in $q$.  For functions such as $(u_i^{-1} |_\bp)_\infty$, which converges for $|q|<1$, the operation $q\mapsto q^{-1}$ will result in a function which converges for $|q|>1$. Therefore, it is necessary to understand the meaning of $\widetilde I(\bp)$ under the mirror map. 

We are now going to regard the modified $I$-function with effective level $ \widetilde I^{\eff} (\bp) $ also as a formal power series with respect to equivariant parameters. More precisely, we treat it as in the following larger space of functions:
\begin{equation} \label{double-series}
\widetilde I^{\eff} (\bp) \quad \in \quad e^{-\sum_{i\not\in \bp} \frac{\ln z_i \ln U_i |_\bp }{\ln q} } \cdot \bC(q) \llbracket a^{ \cA(\bp)^\vee},  Q^{\Eff(\bp)} \rrbracket , 
\end{equation}
where $a^{\cA(\bp)^\vee}$ here means the power series consists of monomials of the form $\prod_{i=1}^n a_i^{l_i}$, which are themselves \emph{effective} equivariant parameters, such that the vector $(l_1, \cdots, l_n)$ lies in the image of  $\cA(\bp)^\vee$ under the map $\beta^T$. Moreover, according to the combinatorial description of $\cA(\bp)$ in Lemma \ref{attracting-cone}, we see that
$$
U_i |_\bp \ \in \ \bC \llbracket a^{\cA(\bp)^\vee} \rrbracket, \qquad i\not\in \bp. 
$$
So the embedding (\ref{double-series}) in fact simply means to expand functions such as $1 / (U_i |_\bp)$ and $1 / (q^{-1} U_i |_\bp; q^{-1})_{D_i}$ for $i\not\in \bp$ as formal power series in $U_i |_\bp$, $i\not\in \bp$. 

\begin{Example} \label{example-for-tau} 
We look at a simplest example to see what the above means. The function $1 / (x)_\infty$, can be expanded as a power series in $x$ by the $q$-binomial formula:
$$
\frac{1}{(x)_\infty} = \sum_{d=0}^\infty \frac{x^d}{(q)_d} \quad \in \quad \bC(q) \llbracket x \rrbracket. 
$$
Assume that the mirror map $\tau$ acts trivially on $x$. Now by our definition, the operation $q\mapsto q^{-1}$ applies to the coefficients. So
$$
\tau \Big( \frac{1}{(x)_\infty} \Big) = \sum_{d=0}^\infty \frac{x^d}{(q^{-1}; q^{-1})_d} \quad \in \quad \bC(q) \llbracket x \rrbracket . 
$$
A second application of the $q$-binomial formula implies that 
$$
\tau \Big( \frac{1}{(x)_\infty} \Big) = (qx)_\infty. 
$$
\end{Example}

Now the mirror map $q\mapsto q^{-1}$ makes sense for formal power series with coefficients rational in $q$, such as in (\ref{double-series}). Our main theorem is the following.

\begin{Theorem}\label{main-theorem}
Let $(\fX, \fX^!)$ be a mirror pair of toric stacks. Let $q^{z_i \partial_{z_i}}$ denote the $q$-difference operator that shifts $z_i \mapsto q z_i$, and similar with $q^{a_i \partial_{a_i}}$. 

\begin{enumerate} [1)]
\setlength{\parskip}{1ex}

\item The modified $I$-function of $\fX$ with effective level structure satisfies the following two sets of $q$-difference equations, with respect to the K\"ahler and equivariant parameters. 

\begin{itemize} 
	\setlength{\parskip}{1ex}

\item Let $\{ e_i \}_{i=1}^n$ be the standard basis of $\bZ^n$, and consider any $\sum_{i=1}^n \mu_i e_i \in  \ker \beta $ such that $\mu_i = \pm 1  $ or $0$.  Denote by $S_\pm$ the subset of indices with $\mu_i= \pm 1$. Then  
\begin{equation} \label{eqn-for-kalher}
\left[ \prod_{i\in S_+} ( z_i^{-1} (1- q^{- z_i \partial_{z_i}} ) ) - \prod_{i\in S_-} (z_i^{-1} (1- q^{- z_i \partial_{z_i}}  ) ) \right] \widetilde I^{\eff} (\fX) = 0. 
\end{equation}		

\item Let $\{ e^!_i \}_{i=1}^n$ be the standard basis of $\bZ^n$ in the dual exact sequence, and consider any $\sum_{i=1}^n \mu^!_i e^!_i \in  \ker \iota^T $ such that $\mu^!_i = \pm 1  $ or $0$. Denote by $R_\pm$ the subset of indices with $\mu^!_i=\pm 1$. Then 
\begin{equation} \label{eqn-for-equiv}
\left[ \prod_{i\in R_+} (a_i^{-1} (1- q^{ a_i \partial_{a_i}} ) ) -  \prod_{i\in R_-} ( a_i^{-1} (1- q^{ a_i \partial_{a_i}} ) ) \right] \left( e^{\sum_{i=1}^n \frac{\ln z_i \ln a_i}{\ln q} } \cdot \widetilde I^{\eff} (\fX) \right) = 0 . 
\end{equation}	

\end{itemize}
Moreover, the solution to the above difference equations is unique, with certain prescribed asymptotic initial condition (see Lemma \ref{asymptotic}). 

\item Under the mirror map 
$$
\tau (z_i^!) = a_i, \qquad \tau (a_i^!) = z_i, \qquad \tau (q) = q^{-1},  
$$
the two sets of $q$-difference equations (\ref{eqn-for-kalher}) (\ref{eqn-for-equiv}), for modified $I$-functions of the mirror pair $(\fX, \fX^!)$ with the effective level structure, coincide with each other. Therefore, combining it with the uniqueness result, we have
$$
\widetilde I^{\eff}(\fX) = e^{\sum_{i=1}^n \frac{\ln z_i \ln a_i}{\ln q} } \cdot \tau ( \widetilde I^{\eff} (\fX^!) )	. 
$$

\end{enumerate}

\end{Theorem}

The rest of this section is to prove the main theorem. Before the proof of the general case, we give a direct computation in the special case of projective spaces.

\subsection{Special case: $\bP^N$} 
Let's consider the following exact sequence
\begin{equation} 
\xymatrix{
	1 \ar[r] & \bC^* \ar[r]^-\iota & (\bC^*)^{N+1} \ar[r]^-\beta & (\bC^*)^{N} \ar[r] & 1,
}  \label{exact-seq-P^N}
\end{equation}
where $\iota$ and $\beta $ are given as follows
\begin{align*}
\iota = \left( \begin{array}{ccccc} 1   \\ 1    \\ \vdots  \\ 1 \end{array}  \right)_{(N+1) \times 1},	\qquad 
\beta = \left( \begin{array}{ccccc} 1& 0 & \cdots & 0  & -1  \\ 0 & 1 & \cdots & 0&-1   \\ \vdots & \vdots & \ddots & \vdots & \vdots \\0 & 0 & \cdots &1 & -1 \end{array}  \right)_{N \times (N+1)} . 
\end{align*}
Then the GIT quotient with respect to $\iota $ is projective space $\bP^{N}$ and $\beta$ gives the mirror. There is only one GIT chamber for $(\bC^*)^{N+1}/\!/_\theta \bC^* $, whose GIT quotient is $\bP^N$. The $(\bC^*)^{N+1}$--fixed points are given by $\bp_j := \{ j\} \subset \{1, \cdots, N+1 \}  $. It's well known that its $I$-function is  \cite{Giv-perm-II}
\begin{align*}
	I_{\mathbb{P}^N}=\sum_{d \geq 0} \frac{Q^d}{\prod_{k=1}^{d}\prod_{i=1}^{N+1}(1-U_i^{-1}q^k)} . 
\end{align*}
The restriction to one of the $(\mathbb{C}^*)^{N+1}$-fixed points is
$$
	I_{\mathbb{P}^N}|_{\bp_j} =\sum_{d \geq 0} \frac{Q^d}{\prod_{i=1}^{N+1}(qU_i^{-1}|_{\bp_j};q)_d }  , 
$$
where
\begin{align*}
U_i^{-1}|_{\bp_j}  = \left\{ \begin{aligned}
	& 1, \qquad && i\in \bp_j \\
	& a_j/a_i . \qquad && i\not\in \bp_j
	\end{aligned} \right. 	
\end{align*}
Then the restriction of $I$-function with effective level structure to the fixed point $\bp_j$ is 
\begin{align*}
I^{\eff}_{\mathbb{P}^N}|_{\bp_j}&=\sum_{d \geq 0} \frac{Q^d}{\prod_{i=1}^{N+1}(q^{-1}U_i|_{\bp_j};q^{-1})_d } \\
&=\sum_{d \geq 0} \frac{(z_1\cdots z_{N+1})^d}{\prod_{i=1}^{N+1}(q^{-1}a_i/a_j;q^{-1})_d}	. 
\end{align*}

Let's consider the mirror of projective space $\mathbb{P}^n$, i.e. the GIT quotient 
\begin{align*}
	\mathbb{C}^{n+1}/\!/_{\theta} (\mathbb{C}^*)^n
\end{align*} 
with charge matrix coming from the dual exact sequence of (\ref{exact-seq-P^N}), i.e.
\begin{equation} 
\xymatrix{
	1 \ar[r] & (\bC^*)^N \ar[r]^-{\iota^{!}} & (\bC^*)^{N+1} \ar[r]^-{\beta^{!}} & \bC^* \ar[r] & 1,
}  
\end{equation}
where
\begin{align*}
\iota^{! } = \beta^T=  \left( \begin{array}{ccccc} 1 & 0 & \cdots & 0    \\ 0 & 1 &   \cdots & 0    \\ \vdots & \vdots & \ddots & \vdots  \\ 0 & 0 &  & 1 \\  -1 & -1 & \cdots & -1  \end{array}  \right)_{(N+1) \times N}	. 
\end{align*}
Let $v_j= (0, \cdots, 1,\cdots,0)$, $j=1,\cdots,N$ be the standard basis of $\bR^N$, and let $v_{N+1}=(-1,-1,\cdots,-1)$. Using the Hilbert--Mumford criterion, we find there are $N+1$ generic chambers
$$
C_j := \bR_{\geq 0} v_1 + \cdots +\bR_{\geq 0}\widehat{v_j}+ \cdots + \bR_{\geq 0} v_{N+1}  . 
$$
For example, the following picture is the chamber structures of $\bP^2$ and its mirror.
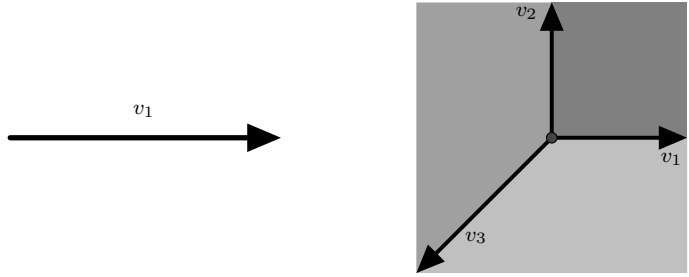
\begin{align*}
\begin{tikzpicture}[line cap=round,line join=round,>=triangle 45,x=1.8cm,y=1.8cm]
\draw [color=cqcqcq,, xstep=0.5cm,ystep=0.5cm] (-5.30540795946268,-1.419981614718684);
\clip(-4.60540795946268,-1.419981614718684) rectangle (2.9187966264045735,1.781117965598392);
\fill[line width=2pt,color=yqyqyq,fill=yqyqyq,fill opacity=1] (0,1) -- (0,0) -- (1,0) -- (1,1) -- cycle;
\fill[line width=2pt,color=aqaqaq,fill=aqaqaq,fill opacity=1] (0,1) -- (-1,1) -- (-1,-1) -- (0,0) -- cycle;
\fill[line width=2pt,color=cqcqcq,fill=cqcqcq,fill opacity=1] (-1,-1) -- (1,-1) -- (1,0) -- (0,0) -- cycle;
\draw [->,line width=2pt] (-4,0) -- (-2,0);
\draw [->,line width=1.5pt] (0,0) -- (1,0);
\draw [->,line width=1.5pt] (0,0) -- (0,1);
\draw [->,line width=1.5pt] (0,0) -- (-1,-1);
\begin{scriptsize}
\draw[color=black] (-3.0111538487800438,0.19915013854564284) node {$v_{1}$};
\draw [fill=uuuuuu] (0,0) circle (2pt);
\draw[color=black] (0.8859117551108426,-0.15764689723853465) node {$v_1$};
\draw[color=black] (-0.18535952561240973,0.9245779349409153) node {$v_2$};
\draw[color=black] (-0.5577311139628201,-0.7371578312203483) node {$v_3$};
\draw[color=yqyqyq] (0.4810218222390622,0.6293456922219092) node {$C_1$};
\draw[color=aqaqaq] (-0.5480734238100464,0.3509838633725603) node {$C_2$};
\draw[color=cqcqcq] (0.21953040725937067,-0.43349038156651337) node {$C_3$};
\end{scriptsize}
\end{tikzpicture}	
\end{align*}
The GIT quotient for each $C_j$ is the affine space $\bC$, containing exactly one fixed point
$$
\bp^{!}_j = \{1,\ldots,N+1 \} \backslash \{ j \} . 
$$
 
The restriction of the $I$-function with effective level is 
\begin{align*}
I^{\eff}|_{\bp_j^!} = \sum_{D \in \Eff(\bp_j^!)} \frac{(z^!)^D}{\prod_{i=1}^{N+1}(q^{-1}U_i|_{\bp_j^!;q^{-1}})_{D_i}}	, 
\end{align*}
where
\begin{align*}
\Eff(\bp_j^!) = \{ D_i \geq 0, \ i \neq j, \   D_j = -D_1 - \cdots - \widehat{D_j} - \cdots - D_{N+1} \leq 0 \}	. 
\end{align*}
So
\begin{equation} \label{I_j} 
{I}^{\eff}|_{\bp_j^!} 
=\sum_{\Eff(\bp_j^!)} \frac{\prod_{i=1,i\neq j}^{N+1}(z^!_i/z^!_{j})^{D_i}}{((a^!_1 \cdots a^!_{N+1})q^{-1};q^{-1})_{D_j} \prod_{i=1,i\neq j}^{N+1}(q^{-1}; q^{-1})_{D_i} }	. 
\end{equation}

Recall the $q$-binomial formula.
$$
\frac{(ax)_\infty}{(x)_\infty} = \sum_{m=0}^\infty \frac{(a)_m}{(q)_m} x^m, \qquad |q|<1, \ |x|<1. 
$$
Let $a = 0$, or let $a = b/x$ and let $x \to 0$. Get
\begin{align} 
\frac{1}{(x)_\infty} = \sum_{m=0}^\infty \frac{x^m}{(q)_m} , \qquad  (b)_\infty = \sum_{m=0}^\infty \frac{q^{m(m-1)/2} (-b)^m }{(q)_m} = \sum_{m=0}^\infty \frac{(q^{-1} b)^m }{(q^{-1}; q^{-1} )_m} .  \label{q-binomial-formula}
\end{align}
Then under the mirror map, we have
\begin{align} \label{I_j-mirror}
&(I^{\eff}_{\bP^N}|_{\bp_{j}}  )(z_i \mapsto a_i^!, a_i \mapsto z_i^!, q \mapsto q^{-1}  ) \nonumber \\
&= \sum_{d \geq 0 }\frac{(a_1^!\cdots a_{N+1}^!)^d}{\prod_{i=1}^{N+1}(q z^{!}_{i} / z^{!}_{j};q )_d} \nonumber  \\
&= \frac{1}{\prod_{i \neq j }^{N+1}(qz^!_i/z^!_j ;q )_\infty} \sum_{d \geq 0} \frac{(a_1^!\cdots a_{N+1}^!)^d}{(q;q)_d} \prod_{i=1, i \neq j}^{N+1} (q^{d+1}z^!_i/z^!_j;q)_\infty \nonumber \\
&=\frac{1}{\prod_{i \neq j }^{N+1}(qz^!_i/z^!_j ;q )_\infty} \sum_{d \geq 0} \frac{(a_1^!\cdots a_{N+1}^!)^d}{(q;q)_d} \prod_{i=1,i\neq j}^{N+1} \left( \sum_{D_i \geq 0} \frac{(q^{d}z^!_i/z^!_j )^{D_i}}{(q^{-1};q^{-1})_{D_i}}   \right) \nonumber \\
&=\frac{1}{\prod_{i \neq j }^{N+1}(qz^!_i/z^!_j ;q )_\infty}  \sum_{D \geq 0} \left(\frac{ \prod_{i=1,i\neq j}^{N+1}(z^!_i/z^!_j)^{D_i}}{ \prod_{i=1,i\neq j}^{N+1}(q^{-1};q^{-1})_{D_i}}  \sum_{d \geq 0 } \frac{\left( q^{\sum_{i=1,i\neq j}^{N+1}D_i}a^!_1\cdots a^!_{N+1}\right)^d}{(q;q)_d}  \right) \nonumber \\
&=\frac{1}{(a_1^!\cdots a_{N+1}^!;q)_\infty\prod_{i \neq j }^{N+1}(qz^!_i/z^!_j ;q )_\infty}  \\ \nonumber
&\times \sum_{D_1,\cdots \hat{j} \cdots D_{N+1} \geq 0}\frac{ \prod_{i=1,i\neq j}^{N+1}(z^!_i/z^!_j)^{D_i}}{ \prod_{i=1,i\neq j}^{N+1}(q^{-1};q^{-1})_{D_i}(q^{-1}a^!_1 \cdots a^!_{N+1};q^{-1})_{-D_1-\cdots \hat{j} \cdots-D_{N+1}} }  . 
\end{align}
Comparing the above formula (\ref{I_j-mirror}) with (\ref{I_j}), we know that under the change of K\"ahler parameters with equivariant parameters and also the change $q \mapsto q^{-1} $, the $I$-functions with effective level structure restricting to the corresponding fixed points are the same by multipling a prefactor.

The modified $I$-function with effective level for $\bP^N$ is as follows,
\begin{align*}
\widetilde{I}_{\bP^N}^{\eff}(\bp_j) &= 	e^{-\sum_{i\not\in {\bp_j}} \frac{\ln z_i \ln U_i |_{\bp_j} }{\ln q} } \cdot  \prod_{i\not\in {\bp_j}} \frac{1}{ (U_i |_{\bp_j} )_\infty } \sum_{\Eff(\bp_j^!)} \frac{\prod_{i=1,i\neq j}^{N+1}(z^!_i/z^!_{j})^{D_i}}{((a^!_1 \cdots a^!_{N+1})q^{-1};q^{-1})_{D_j} \prod_{i=1,i\neq j}^{N+1}(q^{-1}; q^{-1})_{D_i} }	.
\end{align*}
In the following, let's compute the prefactor under the mirror map.
\begin{align*}
	\tau( \sum_{i=1, i \neq j}^{N+1} \ln z_i \ln a_i/a_j ) &=  \sum_{i=1, i \neq j}^{N+1} \ln a^!_i \ln z^!_i/z^!_j  \\
	&= \sum_{i=1}^{N+1} \ln a_i^! \ln z_i^! - \ln z_j^! \ln \big( \prod_{i=1}^{N+1} a_i^!  \big),
\end{align*}
and from Example \ref{example-for-tau}, we have
\begin{align*}
\tau \big( \prod_{i=1,i \neq j}^{N+1} \frac{1}{(a_i/a_j)_{\infty}}	\big) = \prod_{i=1, i \neq j}^{N+1} (qz^!_i/z^!_j)_\infty .
\end{align*}
In summary, we obtain
\begin{align*}
\tau (\widetilde{I}_{\bP^N}^{\eff}(\bp_j)) = e^{\sum_{i=1}^{N+1}\frac{\ln z_i^! \ln a_i^!}{\ln q}} 	\widetilde{I}^{\eff}(\bp_j^!)
\end{align*}
 In the following subsection, we consider general cases, and compare the modified $I$-function with level structure.

\subsection{Proof of the main theorem}
In this subsection we prove our main Theorem \ref{main-theorem}. We briefly summarize the structure of the proof. First, we use the explicit formula of the modified $I$-function to find the $q$-difference equations. We then show the uniqueness of the $q$-difference equations in Lemma \ref{asymptotic}, which uniquely characterizes the modified $I$-functions. Finally, we compare the modified $I$-functions of a mirror pair, under the mirror map, and identify them. 

\begin{proof}[Proof of Theorem \ref{main-theorem} 1): $q$-difference equations] 
By definition of the modified $I$-function with effective level structure, it suffices to prove this for the contribution from each fixed point $\widetilde{I}^{\eff}(\bp)$.  Let $\bp \in \fX^\sT$ be a fixed point. Let's compute the actions of $q$-difference operators on the $I$-function. Recall that by Remark \ref{sum-over-lattice}), 
$$
\widetilde I^{\eff} (\bp) = e^{-\sum_{i\not\in \bp} \frac{\ln z_i \ln U_i |_\bp }{\ln q} }  \cdot  \prod_{i\not\in \bp} \frac{1}{ (U_i |_\bp )_\infty } \cdot \sum_{d\in \bZ^k} \frac{z^D}{\prod_{i=1}^n (q^{-1} U_i |_\bp; q^{-1} )_{D_i} } . 
$$
Therefore for all $1\leq i\leq n$,
$$
q^{- z_i \partial_{z_i}}  \widetilde I^{\eff} (\bp) = e^{-\sum_{i\not\in \bp} \frac{\ln z_i \ln U_i |_\bp }{\ln q} } \cdot  \prod_{i\not\in \bp} \frac{1}{ (U_i |_\bp )_\infty } \cdot \sum_{d\in \bZ^k} \frac{z^D}{\prod_{i=1}^n (q^{-1} U_i |_\bp; q^{-1} )_{D_i} } \cdot q^{-D_i} U_i |_\bp  . 
$$

To compute the action of $q^{-a_i \partial_{a_i}}$'s, we need an explicit expression of $U_i |_\bp$ in terms of $a_i$'s. Let $P$ be the $k\times k$ submatrix of $\iota$ with rows in $\bp$, which is of full rank, and let $Q$ be the $d\times k$ submatrix of $\iota$ with rows not in $\bp$. Let $C:= Q P^{-1}$. The unique solution to the system (\ref{solution}) is then
$$
U_i |_\bp = \left\{ \begin{aligned}
& 1 , && \qquad i\in \bp \\
& a_i \prod_{j\in \bp} a_j^{-C_{ij}}, && \qquad i\not\in \bp. 
\end{aligned} \right. 
$$
We see that for $i\not\in \bp$, 
$$
q^{ a_i \partial_{a_i}}  \widetilde I^{\eff} (\bp) = e^{-\sum_{i\not\in \bp} \frac{\ln z_i \ln U_i |_\bp }{\ln q} }  \cdot  \prod_{i\not\in \bp} \frac{1}{ (U_i |_\bp )_\infty } \cdot \sum_{d\in \bZ^k} \frac{z^D}{\prod_{i=1}^n (q^{-1} U_i |_\bp; q^{-1} )_{D_i} } \cdot z_i^{-1} (1 - q^{-D_i} U_i |_\bp ), 
$$
and for $j\in \bp$, 
\ben
q^{ a_j \partial_{a_j}}  \widetilde I^{\eff} (\bp) &=& e^{-\sum_{i\not\in \bp} \frac{\ln z_i \ln (q^{-C_{ij}} U_i |_\bp) }{\ln q} } \cdot  \prod_{i\not\in \bp} \frac{1}{ (q^{-C_{ij}} U_i |_\bp )_\infty } \\
&& \cdot \sum_{d\in \bZ^k} \frac{z^D}{\prod_{i\in \bp} (q^{-1}; q^{-1})_{D_i} \prod_{i\not\in\bp} (q^{-1 - C_{ij} } U_i |_\bp; q^{-1} )_{D_i} } \\
&=& e^{-\sum_{i\not\in \bp} \frac{\ln z_i \ln U_i |_\bp }{\ln q} } \cdot  \prod_{i\not\in \bp} \frac{1}{ (U_i |_\bp )_\infty } \cdot \sum_{d\in \bZ^k} \frac{z^D}{\prod_{i\in \bp} (q^{-1}; q^{-1})_{D_i} \prod_{i\not\in\bp} (q^{-1} U_i |_\bp; q^{-1} )_{D_i + C_{ij}}  } \\
&& \cdot \prod_{i\not\in \bp} \left( z_i^{C_{ij}} (U_i |_\bp)_{-C_{ij}} (q^{-1} U_i |_\bp ; q^{-1} )_{C_{ij}} \right) \\
&=& e^{-\sum_{i\not\in \bp} \frac{\ln z_i \ln U_i |_\bp }{\ln q} } \cdot  \prod_{i\not\in \bp} \frac{1}{ (U_i |_\bp )_\infty } \cdot \sum_{d\in \bZ^k} \frac{\prod_{i\in \bp} z_i^{D_i} \prod_{i\not\in \bp} z_i^{D_i + C_{ij} } }{\prod_{i\in \bp} (q^{-1}; q^{-1})_{D_i} \prod_{i\not\in\bp} (q^{-1} U_i |_\bp; q^{-1} )_{D_i + C_{ij}}  } \\ 
&=& e^{-\sum_{i\not\in \bp} \frac{\ln z_i \ln U_i |_\bp }{\ln q} } \cdot  \prod_{i\not\in \bp} \frac{1}{ (U_i |_\bp )_\infty } \\
&& \cdot \sum_{d\in \bZ^k} \frac{z_j^{D_j + 1} \prod_{i\in \bp \backslash\{j\} } z_i^{D_i} \prod_{i\not\in \bp} z_i^{D_i + C_{ij} }  \cdot z_j^{-1} (1 - q^{-D_j - 1} )}{(q^{-1}; q^{-1} )_{D_j+1} \prod_{i\in \bp \backslash \{j\} } (q^{-1}; q^{-1})_{D_i} \prod_{i\not\in\bp} (q^{-1} U_i |_\bp; q^{-1} )_{D_i + C_{ij}}  }  \\
&=& e^{-\sum_{i\not\in \bp} \frac{\ln z_i \ln U_i |_\bp }{\ln q} } \cdot  \prod_{i\not\in \bp} \frac{1}{ (U_i |_\bp )_\infty } \cdot \sum_{d\in \bZ^k} \frac{z^D}{\prod_{i=1}^n (q^{-1} U_i |_\bp; q^{-1} )_{D_i} } \cdot z_j^{-1} (1 - q^{-D_j} ) , 
\een
where we used the identity $(x)_d \cdot (q^{-1} x; q^{-1})_{-d} = 1$. Comparing the results above, we obtain the conclusion that for any $1\leq i\leq n$, the modified $I$-function $\widetilde I (\bp)$ satisfies the following $q$-difference equations:
\begin{align} \label{linear-relation}
\left( q^{- z_i \partial_{z_i}} + z_i q^{ a_i \partial_{a_i}} -1 \right) \widetilde I^{\eff} (\bp) = 0. 
\end{align}

By (\ref{linear-relation}), (\ref{eqn-for-kalher}) is equivalent to the identity
$$
\prod_{i\in S_+} q^{a_i \partial_{a_i}}  \prod_{i\in S_-} q^{-a_i \partial_{a_i}} \widetilde I (\bp) = \widetilde I^{\eff}(\bp), 
$$
which follows from the fact that $\widetilde I(\bp)$ depends only on $a_i$'s in terms of $U_i |_\bp$, and hence only depends on the \emph{effective} equivariant parameters. Indeed, the \emph{effective} equivariant parameter is as follows
\begin{align*}
\Lambda_b=\prod_{j=1}^{n} a_j^{\beta_{bj}}, \qquad 1\leq b\leq d. 	
\end{align*}
Then
\begin{align*}
\prod_{i\in S_+} q^{a_i \partial_{a_i}}  \prod_{i\in S_-} q^{-a_i \partial_{a_i}} \Lambda_b = q^{\sum_{i \in S_+ }\beta_{bi} - \sum_{i \in S_{-} }\beta_{bi}} \cdot \prod_{j=1}^n a_j^{\beta_{bj}}		. 
\end{align*}
 Let $\{ e_i \}_{i=1}^n$ be the standard basis of $\bZ^n$, and consider any $\sum_{i=1}^n \mu_i e_i \in  \ker \beta $ such that $\mu_i = \pm 1  $ or $0$. Denote by $S_\pm$ the subset of indexes that $\mu_i= \pm 1$. We know that $\sum_{i \in S_+ }\beta_{bi} - \sum_{i \in S_{-} }\beta_{bi} = 0$, and hence $q^{\sum_{i \in S_+ }\beta_{bi} - \sum_{i \in S_{-} }\beta_{bi}}=1$. 

For (\ref{eqn-for-equiv})  , note that $e^{\sum_{i=1}^n \frac{\ln z_i \ln a_i}{\ln q} } \cdot \widetilde I^{\eff} (\bp)$ satisfies an analogue to Proposition \ref{linear-relation}:
$$
\left( a_iq^{- z_i \partial_{z_i}} +  q^{ a_i \partial_{a_i}} -1 \right)  \left( e^{\sum_{i=1}^n \frac{\ln z_i \ln a_i}{\ln q} } \cdot \widetilde I^{\eff} (\bp) \right) = 0. 
$$
Moreover, we have
$$
\sum_{i=1}^n \ln z_i \ln a_i - \sum_{i\not\in\bp} \ln z_i \ln u_i |_\bp = - \sum_{j\not\in p} \ln \Big( z_j^{-1} \prod_{i\in \bp} z_i^{-C_{ij}} \Big) \ln a_j , 
$$
and hence $e^{\sum_{i=1}^n \frac{\ln z_i \ln a_i}{\ln q} } \cdot \widetilde I (\bp)$ depends only on the \emph{effective} K\"ahler parameters. (\ref{eqn-for-equiv}) then follows from similar arguments as (\ref{eqn-for-kalher}) does. 
\end{proof}

We can also deduce the $q$-difference equations satisfied by the $I$-functions without modification. 

\begin{Corollary}
Let $\{ e_i \}_{i=1}^n$ be the standard basis of $\bZ^n$, and consider any $\sum_{i=1}^n \mu_i e_i \in  \ker \beta $ such that $\mu_i = \pm 1  $ or $0$. Denote by $S_\pm$ the subset of indexes that $\mu_i=\pm 1$. Then the $I$-function (without modification) $I^{\eff} (\bp)$ satisfies
\begin{equation} \label{q-diff-z-2}
\left[ \prod_{i\in S_+} ( z_i^{-1} (1- U_i |_\bp \cdot q^{- z_i \partial_{z_i}} ) ) - \prod_{i\in S_-} (z_i^{-1} (1- U_i |_\bp \cdot q^{- z_i \partial_{z_i}}  ) ) \right] I^{\eff} (\bp) = 0. 
\end{equation}
\end{Corollary}

To study the uniqueness of the $q$-difference equations, we need the following lemma. 

\begin{Lemma} \label{FG-uniqueness}
Let $K$ be a field. Suppose $f (x_1, \cdots, x_k) \in K(q) \llbracket x_1, \cdots, x_k \rrbracket$ satisfies the following system of $q$-difference equations
$$
\left[ F_j (q^{x_1 \partial_{x_1} }, \cdots, q^{x_k \partial_{x_k} } ) - x_j G_j (q^{x_1 \partial_{x_1} }, \cdots, q^{x_k \partial_{x_k} } ) \right] f (x_1, \cdots, x_k) = 0, \qquad 1\leq j\leq k,  
$$
where $F_j$ and $G_j$'s are polynomials with coefficients in $R(q)$, such that $F_j (q^{n_1}, \cdots, q^{n_k}) \neq 0$ for any $n_1, \cdots, n_k \in \bZ_{\geq 0}$, $n_j >0$. Then the solution $f (x_1, \cdots, x_k) \in R(q) \llbracket x_1, \cdots, x_k \rrbracket$ is uniquely determined by its constant term $f(0, \cdots, 0)$. 
\end{Lemma}

\begin{proof}
Let $f$ be 
$$
f(x_1, \cdots, x_k) = \sum_{n_1, \cdots, n_k\geq 0} f_{n_1, \cdots, n_k} x_1^{n_1} \cdots x_k^{n_k}, \qquad f_{n_1, \cdots, n_k} \in K(q). 
$$
For any $1\leq j\leq k$, the $j$-th $q$-difference equation implies the following recursion relations among the coefficients:
$$
F_j (q^{n_1}, \cdots, q^{n_k} ) f_{n_1, \cdots, n_k} = G_j (q^{n_1}, \cdots, q^{n_j - 1}, \cdots, q^{n_k} ) f_{n_1, \cdots, n_j - 1, \cdots, n_k}, 
$$
for any $n_j\geq 1$, and other $n_i\geq 0$, $i\neq j$. By the nonvanishing assumption on $F_j (q^{n_1}, \cdots, q^{n_k} )$, the coefficient $f_{n_1, \cdots, n_k}$ is determined by $f_{n_1, \cdots, n_j - 1, \cdots, n_k}$. The entire formal power series $f$ is then uniquely determined by $f_{0, \cdots, 0}$. 
\end{proof}

\begin{Lemma} \label{asymptotic}

Let $\bp\in \fX$ be a $\sT$-fixed point. 
	
\begin{enumerate}[1)]

\item The $I$-function $I^{\eff} (\bp)$ is uniquely characterized as the solution to the system of $q$-difference equations (\ref{q-diff-z-2}), taking values in $K_{\sT \times \bC_q^*} (\bp)_\loc \llbracket Q^{\Eff(\bp)} \rrbracket  $, and satisfying the initial condition $I^{\eff}(\bp) |_{Q=0} = 1$. 

\item $\widetilde I^{\eff}(\bp)$ is uniquely characterized as the solution to the system of $q$-difference equations (\ref{eqn-for-kalher}), with the following prescribed asymptotic behavior:
$$
\widetilde I^{\eff} (\bp) \quad \in \quad e^{-\sum_{i\not\in \bp} \frac{\ln z_i \ln u_i |_\bp }{\ln q} } \cdot \prod_{i\not\in \bp} \frac{1}{ (U_i |_\bp )_\infty } \cdot \left( 1+ Q \cdot K_{\sT \times \bC^*_q} (\bp)_\loc \llbracket Q^{\Eff(\bp)} \rrbracket  \right) , 
$$
where $Q \cdot K_{\sT \times \bC^*_q} (\bp)_\loc \llbracket Q^{\Eff(\bp)} \rrbracket$ denotes the maximal ideal in $K_{\sT \times \bC^*_q} (\bp)_\loc \llbracket Q^{\Eff(\bp)} \rrbracket$ generated by monomials $Q^\beta$, with $\beta \in \Eff (\bp) \backslash \{0\}$.

\end{enumerate}
 
\end{Lemma}

\begin{proof}
It suffices to prove 1), which 2) directly follows from. Let $\bp\in \fX$ be a fixed point. Up to a change of basis on $\bZ^k$, we can assume that the matrix $\iota$ is of the form $\begin{pmatrix}
I \\
C
\end{pmatrix}$, where $I$ is the $k\times k$ submatrix with rows in $\bp$. To avoid complicated notations, we assume that $\bp = \{1, \cdots, k\}$, without the loss of too much generality. We have $Q_j = z_j \prod_{i=k+1}^n z_i^{C_{ij}}$, $1\leq j\leq k$. Each column of $\iota$ then gives a circuit, and the $q$-difference equations (\ref{q-diff-z-2}), written in terms of $Q_j$'s, are
\ben
&& \left[ (1 - q^{- Q_j \partial_{Q_j}} )  \prod_{i\not\in \bp, \, C_{ij} = 1} (1- U_i |_\bp \cdot q^{- \sum_{l\in \bp} C_{il} Q_l \partial_{Q_l}} ) \right. \\
&& \left. - Q_j \prod_{i\not\in \bp, \, C_{ij} = -1} (1- U_i |_\bp \cdot q^{- \sum_{l\in \bp} C_{il} Q_l \partial_{Q_l}}  )  \right]  I^{\eff} (\bp) = 0
\een
for any $1\leq j\leq k$. We used the fact $U_j |_\bp = 1$ since $j\in \bp$; and note that $U_i |_\bp \neq 1$, for all $i\not\in \bp$. The system of $q$-difference equations then satisfies the assumption in Lemma \ref{FG-uniqueness}, and hence the lemma follows. 
\end{proof}

Lemma \ref{asymptotic} proves the uniqueness part in 1) of Theorem \ref{main-theorem}. Now let's prove the second part of Theorem \ref{main-theorem}.

\begin{proof}[Proof of Theorem \ref{main-theorem} 2)]
From (\ref{eqn-for-equiv}), we know that $ \widetilde I (\bp^!) $ satisfies
\begin{align*}
\left[ \prod_{i\in R_+} ((a_i^!)^{-1} (1- q^{ a^!_i \partial_{a^!_i}} ) ) -  \prod_{i\in R_-} ( (a^!_i)^{-1} (1- q^{ a^!_i \partial_{a^!_i}} ) ) \right] \left( e^{\sum_{i=1}^n \frac{\ln z^!_i \ln a^!_i}{\ln q} } \cdot \widetilde I^{\eff} (\bp^!) \right) = 0 	. 
\end{align*}
Applying the mirror map $\tau$ to both sides, we can see that the two functions
$$
\widetilde I(\bp), \qquad  e^{-\sum_{i=1}^n \frac{\ln z_i \ln a_i}{\ln q} } \cdot \tau ( \widetilde I^{\eff} (\bp^!) )
$$
satisfy the same $q$-difference equations. We regard the two functions
as formal power series in K\"ahler parameters, with appropriate exponential prefactors. Therefore, by the uniqueness result, it suffices to check that $e^{\sum_{i=1}^n \frac{\ln z_i \ln a_i}{\ln q} } \cdot \tau ( \widetilde I^{\eff} (\bp^!) )$ admits the same asymptotic form as in 2) of Lemma \ref{asymptotic}. 

Recall that
$$
\widetilde I^{\eff} (\bp^!) = e^{-\sum_{i\not\in \bp^!} \frac{\ln z_i^! \ln U_i |_{\bp^!} }{\ln q} } \cdot  \prod_{i\not\in \bp^!} \frac{1}{ (U_i |_{\bp^!} )_\infty } \cdot \sum_{d^! \in \Eff(\bp^!)} \frac{(z^!)^{D^!}}{\prod_{i\in \bp^!} (q^{-1}; q^{-1})_{D_i^!} \prod_{i\not\in \bp^!} (q^{-1} U_i |_{\bp^!}; q^{-1} )_{D_i^!} } . 
$$
Its asymptotes as $U_i |_{\bp^!} \to 0$ is
$$
e^{-\sum_{i\not\in \bp^!} \frac{\ln z_i^! \ln U_i |_{\bp^!} }{\ln q} } \cdot \sum_{D_i^! \geq 0, \, i\in \bp^!} \frac{(z^!)^{D^!}}{\prod_{i\in \bp^!} (q^{-1}; q^{-1})_{D_i^!}  } , 
$$
where the summation $D_i^!\geq 0$ for $i\in \bp^!$ is by the combinatorial description of $\Eff(\bp^!)$ in Lemma \ref{key-lemma}. 

Without loss of generality, we assume that the matrix $\iota$ is of the form $\begin{pmatrix}
I \\
C
\end{pmatrix}$, and the matrix $\beta$ is of the form $\begin{pmatrix}
-C^T & I 
\end{pmatrix}$. Denoted by $C_{ij}$ the entries of submatrix $C$ in $iota$ and denoted by $C_{ij}^!$ the entries of submatrix $-C^T$ in $\beta$, then we have $C_{ji}^! = - C_{ij}$, for $j\in \bp$, $i\not\in \bp$. So
\begin{align*}
U_j |_{\bp^!} = a_j^! \prod_{i\in \bp^!} (a_i^!)^{C_{ij}}  \quad {\rm{and}}  \quad (z^!)^{D^!} = \prod_{j\in \bp^!} (z_j^!)^{D_j^!} \prod_{j \not\in \bp^!} (z_j^!)^{\sum_{i\in \bp^!} C_{ji}^! D_i^!}
\end{align*}
Under the mirror map, we have
\ben
\tau \Big( \sum_{j\not\in \bp^!} \ln z_j^! \ln U_j |_{\bp^!} \Big) &=& \tau \Big( \sum_{j\not\in \bp^!} \ln z_j^! \ln \Big( a_j^! \prod_{i\in \bp^!} (a_i^!)^{C_{ij}} \Big) \Big) \\
&=& \sum_{j\in \bp} \ln a_j \ln \Big( z_j \prod_{i\not\in \bp} (z_i)^{C_{ij}} \Big) \\
&=& \sum_{j=1}^n \ln z_j \ln a_j + \sum_{i \not\in \bp} \ln z_i \ln \Big( a_i^{-1} \prod_{j\in \bp} a_j^{C_{ij}} \Big) \\
&=& \sum_{j=1}^n \ln z_j \ln a_j - \sum_{i \not\in \bp} \ln z_i \ln U_i |_\bp, 
\een
and
\begin{align*}
\tau \left( \sum_{D_i^! \geq 0, \, i\in \bp^!} \frac{(z^!)^{D^!}}{\prod_{i\in \bp^!} (q^{-1}; q^{-1})_{D_i^!}  } 	\right) &= \sum_{D_i \geq 0, i \notin \bp} \frac{\prod_{i \notin \bp}a^{D_i}_i \prod_{j \in \bp} a_j^{\sum_{i \notin \bp}-C_{ij}D_i} }{\prod_{i \notin \bp}(q;q)_{D_i}} \\
&=\sum_{D_i \geq 0, i \notin \bp} \frac{\prod_{i \notin \bp} \left( a_i \prod_{j \in \bp} a_j^{-C_{ij}} \right)^{D_i} }{\prod_{i \notin \bp}(q;q)_{D_i}} \\
&=\prod_{i \notin \bp}\frac{1}{(U_i|_{\bp})_\infty}
\end{align*}
where we use the $q$-binomial formula (\ref{q-binomial-formula}).

Hence as $\tau(U_i|_{\bp^!}) \to 0$, 
\ben
\tau (\widetilde I^{\eff} (\bp^!) ) &\sim&  e^{\sum_{i=1}^n \frac{\ln z_i \ln a_i}{\ln q} } \cdot e^{-\sum_{i\not\in \bp} \frac{\ln z_i \ln u_i |_\bp }{\ln q} } \cdot\prod_{i\not\in \bp} \frac{1}{(U_i |_\bp )_\infty } . 
\een
We see that $e^{-\sum_{i=1}^n \frac{\ln z_i \ln a_i}{\ln q} } \cdot \tau ( \widetilde I^{\eff} (\bp^!) )$ admits the same prefactor as in Lemma \ref{asymptotic} when $Q\to 0$. The theorem is then proved. 
\end{proof}

\begin{Remark}
Here we only treat the modified $I$-functions as \emph{formal} power series in K\"ahler or equivariant parameters. In general they do not converge as (multi-valued) analytic functions. For K\"ahler parameters this is due to the possible divergence of the $I$-function $I^{\eff}(\bp)$ itself. For equivariant parameters, this can be seen already from the formula of $\widetilde I^{\eff}(\bp)$: in terms of $U_i |_\bp$, it admits an infinite family of poles $U_i |_\bp = q^{\bZ}$.  
\end{Remark}

\bibliographystyle{abbrv}
\bibliography{3d_N_2_mirror}

\begin{thebibliography}{10}

\bibitem{AHKT}
M.~Aganagic, K.~Hori, A.~Karch, and D.~Tong.
\newblock {Mirror symmetry in {$2+1$} and {$1+1$} dimensions}.
\newblock {\em J. High Energy Phys.}, (7):Paper 22, 30, 2001.

\bibitem{AOelliptic}
M.~Aganagic and A.~Okounkov.
\newblock {Elliptic stable envelope}.
\newblock {\em arXiv:1604.00423}, 2016.

\bibitem{ARW}
O.~Aharony, S.~S. Razamat, and B.~Willett.
\newblock From 3d duality to 2d duality.
\newblock {\em J. High Energy Phys.}, (11):090, front matter+62, 2017.

\bibitem{BH}
L.~A. Borisov and R.~P. Horja.
\newblock {On the {$K$}-theory of smooth toric {DM} stacks}.
\newblock In {\em Snowbird lectures on string geometry}, volume 401 of {\em
  Contemp. Math.}, pages 21--42. Amer. Math. Soc., Providence, RI, 2006.

\bibitem{CK}
I.~Ciocan-Fontanine and B.~Kim.
\newblock Wall-crossing in genus zero quasimap theory and mirror maps.
\newblock {\em Algebr. Geom.}, 1(4):400--448, 2014.

\bibitem{CKM}
I.~Ciocan-Fontanine, B.~Kim, and D.~Maulik.
\newblock {Stable quasimaps to {GIT} quotients}.
\newblock {\em J. Geom. Phys.}, 75:17--47, 2014.

\bibitem{DH}
I.~V. Dolgachev and Y.~Hu.
\newblock Variation of geometric invariant theory quotients.
\newblock {\em Inst. Hautes \'{E}tudes Sci. Publ. Math.}, (87):5--56, 1998.
\newblock With an appendix by Nicolas Ressayre.

\bibitem{DT}
N.~Dorey and D.~Tong.
\newblock {Mirror symmetry and toric geometry in three dimensional gauge
  theories}.
\newblock {\em J. High Energy Phys.}, (5):Paper 18, 16, 2000.

\bibitem{Ful}
W.~Fulton.
\newblock {\em {Introduction to toric varieties}}, volume 131 of {\em Annals of
  Mathematics Studies}.
\newblock Princeton University Press, Princeton, NJ, 1993.
\newblock The William H. Roever Lectures in Geometry.

\bibitem{Giv-WDVV}
A.~Givental.
\newblock {On the {WDVV} equation in quantum {$K$}-theory}.
\newblock volume~48, pages 295--304. 2000.
\newblock Dedicated to William Fulton on the occasion of his 60th birthday.

\bibitem{Giv-perm-II}
A.~Givental.
\newblock {Permutation-equivariant quantum K-theory II. Fixed point
  localization}.
\newblock 2015.

\bibitem{Giv-perm-V}
A.~Givental.
\newblock {Permutation-equivariant quantum K-theory V. Toric $q$-hypergeometric
  functions}.
\newblock 2015.

\bibitem{GKM}
M.~Goresky, R.~Kottwitz, and R.~MacPherson.
\newblock Equivariant cohomology, {K}oszul duality, and the localization
  theorem.
\newblock {\em Invent. Math.}, 131(1):25--83, 1998.

\bibitem{HL}
M.~Harada and G.~D. Landweber.
\newblock The {$K$}-theory of abelian symplectic quotients.
\newblock {\em Math. Res. Lett.}, 15(1):57--72, 2008.

\bibitem{KZ}
P.~Koroteev and A.~M. Zeitlin.
\newblock {qKZ/tRS Duality via Quantum K-Theoretic Counts}.
\newblock 2018.

\bibitem{Lee}
Y.-P. Lee.
\newblock Quantum {$K$}-theory. {I}. {F}oundations.
\newblock {\em Duke Math. J.}, 121(3):389--424, 2004.

\bibitem{MO}
D.~Maulik and A.~Okounkov.
\newblock {Quantum Groups and Quantum Cohomology}.
\newblock 2012.

\bibitem{MFK}
D.~Mumford, J.~Fogarty, and F.~Kirwan.
\newblock {\em Geometric invariant theory}, volume~34 of {\em Ergebnisse der
  Mathematik und ihrer Grenzgebiete (2) [Results in Mathematics and Related
  Areas (2)]}.
\newblock Springer-Verlag, Berlin, third edition, 1994.

\bibitem{RSVZ}
R.~Rimányi, A.~Smirnov, A.~Varchenko, and Z.~Zhou.
\newblock {3d Mirror Symmetry and Elliptic Stable Envelopes}.
\newblock 2019.

\bibitem{RSVZ2}
R.~Rimányi, A.~Smirnov, A.~Varchenko, and Z.~Zhou.
\newblock Three-dimensional mirror self-symmetry of the cotangent bundle of the
  full flag variety.
\newblock {\em Symmetry, Integrability and Geometry: Methods and Applications},
  Nov 2019.

\bibitem{Rom}
M.~Romagny.
\newblock Group actions on stacks and applications.
\newblock {\em Michigan Math. J.}, 53(1):209--236, 2005.

\bibitem{RZ}
Y.~Ruan and M.~Zhang.
\newblock {The level structure in quantum $K$-theory and mock theta functions}.
\newblock 2018.

\bibitem{SZ}
A.~Smirnov and Z.~Zhou.
\newblock {3d Mirror Symmetry and Quantum $K$-theory of Hypertoric Varieties}.
\newblock 2020.

\bibitem{Tha}
M.~Thaddeus.
\newblock Geometric invariant theory and flips.
\newblock {\em J. Amer. Math. Soc.}, 9(3):691--723, 1996.

\bibitem{TY}
H.-H. Tseng and F.~You.
\newblock {K-theoretic quasimap invariants and their wall-crossing}.
\newblock 2016.

\end{thebibliography}
\end{document}